\documentclass[11pt]{amsart}
\usepackage{graphicx} 
\usepackage[margin=1in]{geometry}
\PassOptionsToPackage{final}{showkeys}
\usepackage{gsty}
\usepackage{thm-restate}

\definecolor{facetdark}{RGB}{120,165,210}
\definecolor{facetlight}{RGB}{190,215,240}

\usepackage{tikz}
\usetikzlibrary{decorations.pathreplacing,calligraphy}
\usetikzlibrary{calc}

\title[Nondegenerate gradient Young measures and dimension reduction]{Nondegenerate gradient Young measures and dimension reduction for nonlinear membranes}

\author{Gokul G. Nair}
\address[G. G. Nair]{Department of Mathematics, University of Michigan, Ann Arbor MI 48109, United States}
\email{ngokul@umich.edu}

\author{David Padilla-Garza}
\address[D. Padilla-Garza]{Institute of Science \& Technology Austria, Am Campus 1, 3400 Klosterneuburg}
\email{David.Padilla-Garza@ist.ac.at}

\author{Marco Picchi Scardaoni}
\address[M. Picchi Scardaoni]{Department of Civil and Industrial Engineering, University of Pisa,
Largo Lucio Lazzarino 1, 56122, Pisa, Italy}
\email{marco.picchiscardaoni@ing.unipi.it}
	
\date{}
\begin{document}
\sloppy

\begin{abstract}
    We consider the rigorous derivation of membrane theories from three-dimensional nonlinear elasticity. A salient feature of our framework is the inclusion of an orientation-preservation constraint together with energy growth that penalizes vanishing local volume. We embed the 3D variational problem into a space of parametrized measures, and obtain a membrane $\Gamma$-limit defined on a class of nondegenerate gradient Young measures. This formulation offers a twofold advantage: first, it uniquely identifies a non-relaxed membrane energy density capable of capturing the fine oscillations (microstructure) of minimizing sequences; second, the resulting membrane density preserves the unbounded energy growth near degenerate configurations. We further show that the classical relaxed membrane energy of Le Dret and Raoult is recovered as the minimum of our functional over all measures with a prescribed barycenter. The main difficulty lies in generating such measures by maps whose gradients have rank two almost everywhere, which we address through piecewise isometric (origami-like) constructions.   
\end{abstract}
\sloppy
\maketitle

\section{Introduction}
The variational derivation of lower-dimensional models from bulk nonlinear elasticity has been extensively studied since the early 1990s, beginning with Acerbi, Buttazzo and Percivale for strings~\cite{acerbi_variational_1991} and Le Dret and Raoult for membranes~\cite{le_dret_nonlinear_1995}. A standard local constraint in bulk nonlinear elasticity, related to non-interpenetrability and nondegeneracy of the volume element, is orientation preservation, expressed by $\det\nabla\bm f>0$ a.e. This is enforced by requiring the stored energy density $W(\bm F)$ to satisfy $W(\bm F)=+\infty$ for $\det\bm F\leq0$ and $W(\bm F)\rightarrow+\infty$ as $\det\bm F\rightarrow0^+$~\cite{ball_convexity_1976}. While this condition was taken into account in~\cite{acerbi_variational_1991}, it was not in~\cite{le_dret_nonlinear_1995} whose hypotheses were incompatible with $W$ taking on the value $+\infty$. This restriction was later removed in Ben Belgacem~\cite{belgacem1997methode,belgacem2000relaxation}, and Anza-Hafsa and Mandallena~\cite{anza_nonlinear_2008} who derived the membrane limit under the $\det\nabla\bm f>0$ constraint. A related dimension reduction result was obtained by Conti and Dolzmann~\cite{conti_derivation_2006} who derived the same limit under the incompressibility constraint $\det\nabla\bm f=1$. In these cases, a key difficulty is the construction of recovery sequences satisfying the pointwise constraint, which requires relying on some approximation theorems for Sobolev maps by immersions. 

Although~\cite{belgacem2000relaxation,anza_nonlinear_2008} impose the constraint $\det\nabla\bm f>0$ at the level of the bulk elasticity model, the membrane energy they identify has the same form as the one obtained by Le Dret and Raoult: in both cases the limiting energy density is $QW_0$, where $W_0:\RR^{3\times2}\rightarrow[0,+\infty]$ is given by $W_0(\bar{\bm F}):=\inf_{\bm z\in\RR^3}W(\bar{\bm F}|\bm z)$ and $Q$ denotes the quasiconvex envelope. This is a relaxed membrane energy, also known as a tension-field model~\cite{pipkin_relaxed_1986}, whose density vanishes on compressive states, representing fine-scale wrinkling. An important feature of this relaxed model is that the effects of the unbounded growth of $W(\bm F)$ as $\det\bm F\rightarrow0^+$ are completely absent in the limiting density, $QW_0$. This singularity is inherited by $W_0$, which is infinite precisely on $3\times2$ matrices with rank less than two, but is destroyed by quasiconvexification. By contrast, ``direct'' shell and membrane models retain this constraint by requiring the energy density to blow up as the local measure of area or volume approaches zero~\cite{ciarlet_orientation-preserving_2013,anicic_polyconvexity_2018,healey_energy_2023,healey2025nonlinearly}.

The disappearance of the constraint is a consequence of the topology in which the limit is computed: weak Sobolev convergence only retains barycentric information about the gradients of minimizing sequences. The limit therefore identifies $QW_0$ and not $W_0$, which is a genuine loss of information as there are infinitely many functions that share the same quasiconvex envelope~\cite{freddi_variational_2008}. To overcome this problem, Freddi and Paroni~\cite{freddi_energy_2004} embedded the three-dimensional energy into a space of parametrized measures and computed the limit with respect to the weak$*$ convergence of measures using the $\Gamma$-convergence machinery developed in~\cite{anzellotti_dimension_1994}. They identified a limiting energy on the space of gradient Young measures that uniquely determines the membrane energy density $W_0$. In that work, however, Freddi and Paroni adopted hypotheses of the same type as Le Dret and Raoult, and did not incorporate the requirement that energy diverge as local volume vanishes. That requirement is incorporated in their companion paper~\cite{freddi_3d_2004}, where the limiting functional for a one-dimensional string is defined on a class of Young measures on $\RR^3$ for which the function $\bm z\mapsto\abs{\bm z}^p+\abs{\bm z}^{-q}$ is integrable --- a condition that prevents the limiting measure from charging degenerate configurations.

In this work, we treat dimension reduction with the orientation-preservation constraint in the Young measure setting. Starting from a bulk energy density $W(\bm F)\rightarrow+\infty$ as $\det\bm F\rightarrow0^+$, we identify the limiting membrane energy density $W_0:\RR^{3\times2}\rightarrow[0,+\infty]$ uniquely, rather than only its quasiconvex envelope. Additionally, we show that $W_0$ inherits the property that $W_0(\bar{\bm F})\rightarrow+\infty$ as $J(\bar{\bm F}):=\sqrt{\det\bar{\bm F}^T\bar{\bm F}}\rightarrow 0$ at a rate determined by the growth conditions of $W$. The limiting problem is posed on a class of gradient Young measures on $\RR^{3\times2}$, which we call \textit{nondegenerate}, for which $\Phi(\bar{\bm F}):=\abs{\bar{\bm F}}^p+\abs{J(\bar{\bm F})}^{-q}$ is integrable (denoted by $\mathcal{JY}^{p,-q}$). As in~\cite{freddi_3d_2004}, this condition prevents the limiting Young measures from charging degenerate gradients. 

The main difficulty in proving this $\Gamma$-convergence result lies in constructing recovery sequences, a task contingent on the answer to the following question: if a gradient Young measure is supported on matrices of rank two, can it be generated by a sequence of gradients that have rank two almost everywhere? Theorem~\ref{thm:gym} answers this question in the affirmative. More precisely, we show that every $\nu\in\mathcal{JY}^{p,-q}$ is generated by a sequence of gradients $\{\nabla\bm u^k\}$ that themselves satisfy the rank constraint and for which $\{\Phi(\nabla\bm u^k)\}$ is equiintegrable.

An important ingredient in the proof of Theorem~\ref{thm:gym} is Lemma~\ref{lemma:replacement-inside-triangle}, which may be of independent interest. It allows us to replace a strictly short affine map $\bm u$ on a triangle $T$ with a piecewise affine map $\tilde{\bm u}$ that agrees with $\bm u$ on $\partial T$ and satisfies quantitative bounds on its Lipschitz constant and $J(\nabla\tilde{\bm u})$. The replacement is obtained by an origami-type (piecewise isometric) construction on a slightly shrunken triangle $T^\delta\subset T$, where $J(\nabla\tilde{\bm u})$ is bounded below by an absolute constant, together with an interpolation back to $\bm u$ on the remaining collar region $T\setminus T^{\delta}$, whose measure is $O(\delta)$. The proof of Theorem~\ref{thm:gym} then proceeds by a sequence of successive modifications -- mollification, approximation by immersions, finite element interpolation, and finally the origami replacement -- each step arranged so as to preserve Young measure generation and equiintegrability.

Origami-based constructions are not new. They appear in the study of crumpling, where Conti and Maggi~\cite{conti2008confining} prove that origami maps are dense in the class of short maps with respect to the uniform norm. Their proof takes the $C^1$-isometric embedding theorems of Nash and Kuiper~\cite{nash1954c,kuiper1955c1,kuiper1955c2} as its starting point. What we require is different: rather than approximating $\bm u$, the replacement $\tilde{\bm u}$ exactly matches it on $\partial T$ and satisfies definite lower bounds on $J(\nabla\tilde{\bm u})$. This is why our approach uses the extension theorem of Brehm~\cite{brehm1981extensions}.

Closer to our setting is the literature on gradient Young measures under determinant constraints. Characterizations have been obtained for gradients with positive determinant in the subcritical range $p<d$~\cite{koumatos2016orientation} and in the plane for quasiregular~\cite{astala2002quasiregular} maps, quasiconformal maps~\cite{benesova2015gradient} and bi-Lipschitz homeomorphisms~\cite{benevsova2016characterization}. It is worth pointing out that Bene\v{s}ov\'a and Kru\v{z}\'{\i}k~\cite{benevsova2016characterization} are also motivated by similar growth hypotheses on $W$ that we assume here. Their result rests on a planar bi-Lipschitz extension theorem~\cite{daneri2011planar} and is confined to the case of $p=\infty$. Our setting differs in that the constraint is local (rather than global injectivity), and the matrices are non-square: the extra codimension gives us the flexibility for a folding-based construction. On the other hand, our result does not apply to the $p=\infty$ case.

Apart from membrane theories, it is also possible to consider other effective (dimensionally reduced) theories for incompressible thin elastic sheets, which correspond to  higher-order energy scaling regimes. Pure-bending theories have also been rigorously derived by $\Gamma$-convergence, building on geometric-rigidity methods; see, for instance, \cite{conti_gamma_2009} for a Kirchhoff plate theory, and \cite{li_kirchhoff_2013} for a variant of the Kirchhoff theory for prestrained shells, including the incompressible case. In \cite{li_chermisi_von_karman_2013}, the authors derive a von K\'arm\'an theory for incompressible elastic shells. Relatedly, Lewicka and Li \cite{lewicka_li_convergence_2015} studied the convergence of equilibria for incompressible elastic shells in the von K\'arm\'an regime.

With Theorem~\ref{thm:gym} in hand, we turn to the dimension-reduction problem. Proposition~\ref{thm:Young-measure-compactness} establishes compactness: a sequence of three-dimensional deformations with bounded energy generates, in a suitable sense, a nondegenerate gradient Young measure $\nu\in\mathcal{JY}^{p,-q}$ whose barycenter is the gradient of the weak Sobolev limit of the sequence. Theorem~\ref{thm:gc} establishes the two halves of the $\Gamma$-convergence result. The lower bound follows from compactness together with the growth properties of $W_0$, while the upper bound relies on the generation result, Theorem~\ref{thm:gym}. Lastly, we prove two additional results: the first, a relaxation result (Proposition~\ref{thm:relaxation}) showing that our limiting functional is the lower semicontinuous envelope of the corresponding energy on Sobolev deformations; the second, a result (Proposition~\ref{prop:relation-to-LDR}) that relates our functional to the one identified by Le Dret and Raoult.

The paper is organized as follows. In Section~\ref{sec:main-results} we introduce the main objects of study and state our principal results. Section~\ref{sec:notation} collects some definitions and standard facts about Young measures needed in the sequel. Section~\ref{sec:young-measure-generation} proves the generation results for nondegenerate gradient Young measures using the origami construction. Section~\ref{sec:problem-formulation-Young} establishes the growth bounds, compactness, and $\Gamma$-convergence for the dimension-reduction problem. Finally, Section~\ref{sec:relaxation} establishes the relaxation result and compares it with the membrane theory of Le Dret and Raoult~\cite{le_dret_nonlinear_1995}.

\section{Main results}\label{sec:main-results}

We now state the main results of this paper. 

\subsection{Nondegenerate Young measures}

We begin with results relating to nondegenerate gradient Young measures. The main result roughly states that a gradient Young measure concentrated on matrices of rank two can be generated by smooth immersions from 2D into 3D. Before stating our results, we recall the definition and basic notions of Young measures. 

Let $\Omega\subset\RR^n$ be an open bounded domain and let $V$ be a finite dimensional vector space (in our applications $V=\RR^{m\times n}$ usually). We denote by $\mathcal{M}(V)$ the space of (finite) positive Radon measures on $V$. From Riesz's representation theorem, we have the well-known duality pairing of $\mu\in\mathcal{M}(V)$ with $\phi\in C_c(V)$:
\begin{align*}
    \innerpdt{\mu,\phi}:=\int_V\phi\dif\mu.
\end{align*}
A \textit{parametrized measure} is a map $\mu:\Omega\rightarrow\mathcal{M}(V)$. For $x\in\Omega$, we often use the notation $\mu_x$ to mean $\mu(x)$.
\begin{definition}
    A parametrized measure $\mu:\Omega\rightarrow\mathcal{M}(V)$ is said to be \textit{weakly* measurable} if for every $\phi\in C_c(V)$ the map $x\mapsto\innerpdt{\mu(x),\phi}$ is Lebesgue measurable.
\end{definition}
The space of parametrized measures for which the map $x\mapsto\innerpdt{\mu(x),\phi}$ is essentially bounded for all $\phi\in C_c(V)$ is denoted by $L^\infty_w(\Omega,\mathcal{M}(V))$. We equip $L^\infty_w(\Omega,\mathcal{M}(V))$ with the weak* topology induced by duality with the space $L^1(\Omega,C_c(V))$. In particular, we say that a sequence $\{\mu^k\}\subset L^\infty_w(\Omega,\mathcal{M}(V))$ converges weakly* to $\mu\in L^\infty_w(\Omega,\mathcal{M}(V))$ if for all $\phi\in C_c(V)$ and all $g\in L^1(\Omega)$
\begin{align*}
    \lim_{k\rightarrow\infty}\int_\Omega\innerpdt{\mu^k_x,\phi} g(x)\dif x=\int_\Omega\innerpdt{\mu_x,\phi} g(x)\dif x.
\end{align*}

\begin{definition}[Young measure]
    A parametrized measure $\mu\in L^\infty_w(\Omega,\mathcal{M}(V))$ is called a \textit{Young measure} if $\mu_x$ is a probability measure for almost every $x\in\Omega$. The space of all such Young measures is denoted $\mathcal{Y}(\Omega,V)$.
\end{definition}

\begin{definition}
    A Young measure $\mu\in\mathcal{Y}(\Omega,V)$ is said to be \textit{generated} by the sequence of measurable functions $\{u^n\}$ if $\delta_{u^n(\cdot)}\xrightharpoonup{*}\mu$ in $L^\infty_w(\Omega,\mathcal{M}(V))$.
\end{definition}
Every Young measure can be generated by a sequence of measurable functions~\cite[Theorem 7.7]{pedregal_parametrized_1997}.

\begin{definition}
    The \textit{barycenter} or center of mass of a Young measure $\mu\in\mathcal{Y}(\Omega,V)$ is the function $x\mapsto\innerpdt{\mu_x,\text{id}}$, where $\text{id}$ is the identity mapping. It will be denoted by $[\mu_{x}]$.
\end{definition} 

\subsection*{Notation}
The definitions above are stated for a general finite-dimensional space $V$ and use lightface type. From now on we work with deformations and adopt a convention common in nonlinear elasticity: lightface type is used for objects associated with the reference configuration, and boldface type for objects taking values in, or mapping into, the deformed configuration. Thus points $x\in\Omega$ and vectors $e_i\in\RR^n$ are written in lightface, whereas maps $\bm u:\Omega\to\RR^m$ (in our applications $n=2$ and $m=3$) and matrices representing linear maps into $\RR^m$, such as gradients $\bm A\in\RR^{m\times n}$, are written in boldface.

In this paper, we work mainly with gradient Young measures:
\begin{definition}[Gradient Young measures]
    A Young measure $\mu\in\mathcal{Y}(\Omega,\RR^{m\times n})$ is called a $p$-gradient Young measure if there exists a sequence $\{\bm u^k\}\subset W^{1,p}(\Omega,\RR^m)$ such that $\bm u^k\weakarrow \bm u$ in $W^{1,p}(\Omega,\RR^m)$ and $\{\nabla \bm u^k\}$ generates $\mu$. In this case, we refer to $\bm u$ as the \textit{underlying deformation} of $\mu$. The space of all $p$-gradient Young measures is denoted by $\mathcal{GY}^p(\Omega,\RR^{m\times n})$.
\end{definition}

We now introduce the class of gradient Young measures central to our analysis. From here on, we specialize to the case $m=3$ and $n=2$ of surfaces deforming in space so that $\Omega\subset\RR^2$ and gradients are in $\RR^{3\times 2}$. Motivated by nonlinear elasticity, we impose an integrability condition that controls both large gradients and degeneration of the area Jacobian. 

\begin{definition}[Nondegenerate gradient Young measures]
    Let $1<p<\infty$ and $0<q<\infty$. We denote by $\mathcal{JY}^{p,-q}$ the set of Young measures $\nu\in\mathcal{GY}^p(\Omega,\RR^{3\times 2})$ such that
    \begin{align*}
        \int_\Omega\int_{\RR^{3\times2}}\Phi(\bm A)\dif\nu_x(\bm A)\dif x<+\infty,
    \end{align*}
    where $\Phi(\bm A):=\abs{\bm A}^p+\abs{J(\bm A)}^{-q}$ and $J(\bm A):=\sqrt{\det \bm A^T\bm A}$.
\end{definition}

We now state the main results of this paper, which concern the generation of a measure $\nu \in \mathcal{JY}^{p,-q}$ by immersions (either smooth or piecewise affine). 

\begin{theorem}\label{thm:gym}
    Let $\Omega\subset\RR^2$ be a bounded Lipschitz domain and $\nu\in\mathcal{JY}^{p,-q}(\Omega,\RR^{3\times2})$. Then there exists a sequence of continuous functions $\{\bm u^k\}\subset W^{1,p}(\Omega,\RR^3)$, piecewise affine on a finite triangulation of $\Omega$, generating $\nu$ and satisfying
    \begin{align*}
        \sup_{k}\int_{\Omega}\Phi(\nabla \bm u^k)\dif x<+\infty.
    \end{align*}
    Furthermore, the sequence $\{\Phi(\nabla\bm u^k)\}$ is equiintegrable, and
    \begin{align*}
        \lim_{k\rightarrow\infty}\int_\Omega\Phi(\nabla\bm u^k)\dif x=\int_\Omega\int_{\RR^{3\times2}}\Phi(\bm F)\dif\nu_x(\bm F)\dif x.
    \end{align*}
\end{theorem}
The proof is given in Section~\ref{sec:young-measure-generation} and the strategy is illustrated in Figure~\ref{fig:replacement-construction}. The idea is to first generate the measure by piecewise affine maps and then to modify them on the ``bad triangles'', those on which the area Jacobian falls below a certain threshold. The key point is Lemma~\ref{lemma:replacement-inside-triangle}, an origami construction that modifies a piecewise affine map inside each bad triangle, without changing its boundary values, so that the area Jacobian is raised above the threshold on most of the triangle. This is what keeps the term $\abs{J(\nabla\bm u^k)}^{-q}$ in $\Phi$ integrable, uniformly in $k$, along the approximating sequence.

\begin{figure}[!t]
  \centering
  \resizebox{\linewidth}{!}{%
    \input{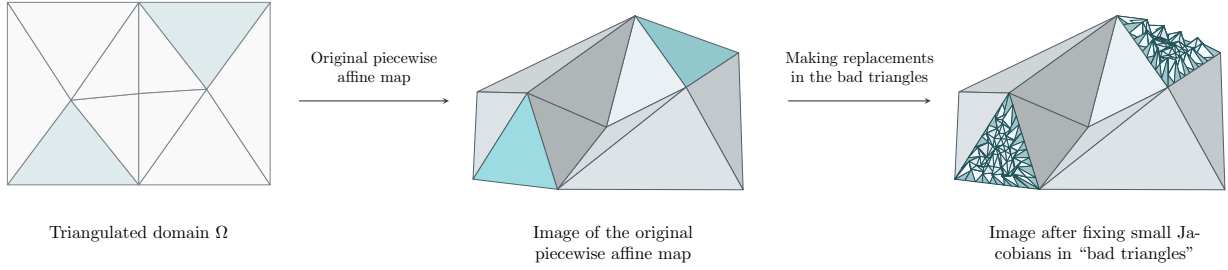}%
  }
  \caption{Schematic representation of the replacement strategy in the proof of Theorem~\ref{thm:gym}. The generating sequence is first approximated by piecewise affine maps. On each ``bad triangle'' (marked in blue), where the Jacobian falls below the threshold, the map is then replaced by the origami-based construction of Lemma~\ref{lemma:replacement-inside-triangle}, which raises the Jacobian above the threshold on most of the triangle.}
\label{fig:replacement-construction}
\end{figure}

To conclude, we state the analogue of Theorem~\ref{thm:gym} for smooth, instead of piecewise affine immersions. 

\begin{proposition}\label{prop:smooth-YM-generation}
    Assume the same hypotheses as Theorem~\ref{thm:gym}. Then $\nu$ can be generated by a sequence $\{\bm u^k\}\subset C^\infty(\bar{\Omega},\RR^3)$ satisfying $J(\nabla\bm u^k)>0$ on $\bar{\Omega}$ with $\{ \nabla \bm u^k\}$ $p$-equiintegrable, such that
    \begin{align}\label{eq:Cinfty-Phi-equiintergability}
        \lim_{k\rightarrow\infty}\int_\Omega\Phi(\nabla\bm u^k)\dif x=\int_\Omega\int_{\RR^{3\times 2}}\Phi(\bm F)\dif\nu_x(\bm F)\dif x.
    \end{align}
\end{proposition}

\subsection{A Young-measure approach to membrane theory}

We now apply the results of the last subsection (Theorem \ref{thm:gym} and Proposition \ref{prop:smooth-YM-generation}) to thin elastic sheets. More specifically, we are interested in the derivation of an effective membrane model for thin elastic sheets under the orientation-preservation constraint. We take a Young-measure approach to this problem, and derive the limiting functional on the space of gradient Young measures. The advantage of this approach is that we retain information about the oscillations in the relaxed regions of the energy. 

Let $\Omega\subset\RR^2$ be an open, bounded, Lipschitz domain and set $\Omega_\e:=\Omega\times\left(-\frac{\e}{2},\frac{\e}{2}\right)\subset\RR^3$. For $x\in\Omega_\e$, we often use the notation $x=(x',x_3)$, where $x'=(x_1,x_2)\in\Omega$.

We first define the class of stored energy functionals that we will treat in this paper. Apart from the standard hypotheses, we impose orientation preservation and growth conditions that penalize vanishing volume. 
\begin{definition}
We consider a 3D (bulk) stored energy density $W:\Omega\times\RR^{3\times 3}\rightarrow[0,+\infty]$,  a function satisfying the following hypotheses:
\begin{enumerate}[(i)]
    \item Regularity: $W$ is a Carath\'eodory integrand, and for a.e. $x'\in\Omega$, $W(x',\cdot)$ is finite on $\{\bm F:\det\bm F>0\}$. 
    \item Orientation preservation:
    \begin{align}
        \label{itm:local-injectivity}
        \det\bm F\leq0\text{ if and only if }W(\cdot,\bm F)=+\infty.\tag{H1}
    \end{align}
    \item For $p,s\in(1,\infty)$, there exist positive constants $C_1$, $C_2$, $C_3$, $C_4$ such that
    \begin{align}\label{itm:dimension-reduction-growth}
        C_1\left(\abs{\bm F}^p+\frac{1}{\abs{\det\bm F}^s}-C_2\right)\leq W(\cdot,\bm F) \leq C_3\left(\abs{\bm F}^p+\frac{1}{\abs{\det\bm F}^s}+C_4\right)\tag{H2}
    \end{align}
    for every $\bm F\in\RR^{3\times3}$ with $\det\bm F>0$.
    \item We also assume that $W$ satisfies material objectivity (frame indifference) i.e., $W(\cdot,\bm Q\bm F)=W(\cdot,\bm F)$ for all $\bm Q\in \text{SO}(3)$.
\end{enumerate}

\end{definition}

We now define the positive-thickness elastic energy functional, in terms of Young measures. 
\begin{definition}
   Let $\e> 0$ and $ \mathcal{E}_\e: L^\infty_w(\Omega_1,\mathcal{M}(\RR^{3\times 3})) \to [0, +\infty]$ be defined as 
\begin{align*}
    \mathcal{E}_\e[\mu]:=
    \begin{cases}
        \int_{\Omega_1} W(x',\nabla_\e\bm f)\dif x&\text{if there exists }\bm f\in\mathcal{A}_\e\text{ s.t. }\mu=\delta_{\nabla_\e\bm f}\\
        +\infty&\text{otherwise in } L^\infty_w(\Omega_1,\mathcal{M}(\RR^{3\times 3})),
    \end{cases}
\end{align*}
where
\begin{equation}
\label{eq:admiss}
    \mathcal{A}_\e:=\left\{\bm f\in W^{1,p}(\Omega_1,\RR^3):\det\nabla_\e\bm f>0\right\},
\end{equation}
and $\nabla_\e:=\left(\diffp{{}}{{x_1}},\diffp{{}}{{x_2}},\frac{1}{\e}\diffp{{}}{{x_3}}\right)$ is the rescaled gradient. 
\end{definition}

Recall that in the first case above, we say that $\bm f$ is the \textit{underlying deformation} of $\mu$. Since $W(x',\cdot)$ and $\delta_{\nabla_\varepsilon\bm f}$ only depend on the gradients, the underlying deformation of any $\mu$ is determined only up to an additive constant. We will select the mean zero representative whenever necessary.

We now introduce the proposed limiting functional:
\begin{definition}
Let $\mathcal{E}_0: L^\infty_w(\Omega,\mathcal{M}(\RR^{3\times2})) \to [0, +\infty]$ be defined as 
    \begin{align*}
				\mathcal{E}_0[\nu]:=\begin{cases}
					\int_\Omega\langle\nu_{x'},W_0(x',\cdot)\rangle\ dx' &\text{if }\nu\in\calJ\calY^{p,-q}(\Omega,\RR^{3\times2}),\\
					+\infty &\text{otherwise in }L^\infty_w(\Omega,\mathcal{M}(\RR^{3\times2})),
				\end{cases}
			\end{align*}
where $W_0:\Omega\times\RR^{3\times2}\rightarrow[0,+\infty]$ is defined by $W_0(x',\bar{\bm F}):=\inf_{\bm z\in\RR^3}W(x',(\bar{\bm F}|\bm z))$ and $q=\frac{ps}{p+s}>0$.
\end{definition}

As a consequence of dimension reduction, the domain for the limiting energy functional will consist of Young measures on a lower dimensional space. These are related to Young measures in the domain of the energy functional with positive thickness by projection operators, which we now define. 
\begin{definition}
    The \textit{through-thickness averaging operator} is defined as 
\begin{align*}
    \mathsf{A}:L^\infty_w\left(\Omega_1,\mathcal{M}(V)\right)\rightarrow L^\infty_w(\Omega,\mathcal{M}(V))\\
    \innerpdt{(\mathsf{A}\nu)_{x'},\phi}:=\int_{-1/2}^{1/2}\innerpdt{ \nu_{(x',x_3)},\phi}\dif x_3\text{ for all }\phi\in C_c(V).
\end{align*}
It is straightforward to check that $\mathsf{A}$ is continuous and maps Young measures to Young measures. Finally we have the \textit{average-projection} operator 
\begin{align*}
    \mathsf{R}:L^\infty_w\left(\Omega_1,\mathcal{M}(\RR^{3\times3})\right)\rightarrow L^\infty_w(\Omega,\mathcal{M}(\RR^{3\times2}))\\
    \mathsf{R}:=\mathsf{P}_{\#}\circ\mathsf{A}=\mathsf{A}\circ\mathsf{P}_{\#},
\end{align*}
where $\mathsf{P}:\RR^{3\times 3}\rightarrow\RR^{3\times 2}$ denotes the \textit{projection operator} from $\RR^{3\times 3}$ onto $\RR^{3\times 2}$, and $\#$ denotes the pushforward measure. 
\end{definition}
Here, the commutativity of the averaging and projection operators can be checked by using their respective definitions. Furthermore, $\mathsf{R}$ is continuous and maps Young measures into Young measures.

We now state the main result of this section: the derivation of an effective membrane model for thin elastic sheets via Young measures. 
\begin{theorem}[$\Gamma$-convergence]\label{thm:gc}
    Fix any $\varepsilon_n\rightarrow0$. As $n\rightarrow +\infty$, the sequence of functionals $\{\mathcal{E}_{\e_n}\}$ $\Gamma$-converges to $\mathcal{E}_0$ in the following sense:
    \begin{enumerate}[(i)]
        \item  (Compactness) Suppose $\{\mu_n\}\subset L^\infty_w(\Omega_1,\mathcal{M}(\RR^{3\times3}))$ is a sequence such that $\sup_n \mathcal{E}_{\e_n}[\mu_n]<+\infty$. Then there exists $\mu\in L^\infty_w(\Omega_1,\mathcal{M}(\RR^{3\times3}))$ and a (not relabelled) subsequence $\{\mu_n\}$ such that
        \begin{align*}
            \mu_n\xrightharpoonup{*}\mu\text{ in }L^\infty_w(\Omega_1,\calM(\RR^{3\times 3})),
        \end{align*}
        and
        \begin{align*}
            \nu:=\mathsf{R}\mu\in\mathcal{JY}^{p,-q}(\Omega,\RR^{3\times2}),
        \end{align*}
        with $q=\frac{ps}{p+s}>0$.
        \item (Liminf inequality) For every $\nu \in L^\infty_w(\Omega, \calM(\bbR^{3\times 2}))$ and for every sequence $\{\mu_{n}\}\subset L^\infty_w(\Omega_1, \calM(\bbR^{3\times 3}))$ such that
        $$
        \mathsf{R}\mu_{n} \weakstar \nu \quad \mbox{in } L^\infty_w(\Omega, \calM(\bbR^{3\times 2}))
        $$
        we have
        \begin{equation*}
            \liminf\limits_{n\rightarrow \infty} \mathcal{E}_{\e_n}[\mu_{n}]\geq \mathcal{E}_0[\nu].
        \end{equation*}
        \item (Recovery sequence)
        For every  $\nu$ in $ L^\infty_w(\Omega, \calM(\bbR^{3\times 2}))$ there exists a sequence $\{\mu_{n}\}\subset L^\infty_w(\Omega_1, \calM(\bbR^{3\times 3}))$ such that
        $$
        \mathsf{R}\mu_{n} \weakstar \nu \quad \mbox{in } L^\infty_w(\Omega, \calM(\bbR^{3\times 2}))
        $$
        and
        \begin{equation*}
            \limsup_{n\rightarrow \infty} \mathcal{E}_{\e_n}[\mu_{n} ]\leq \mathcal{E}_0[\nu].
        \end{equation*}
    \end{enumerate}
\end{theorem}

\subsection{Relation to classical membrane theory}

Finally, we relate our membrane functional $\mathcal{E}_0$ to classical membrane theory. The main result of this subsection is that we can identify the asymptotic local oscillations (i.e. limiting gradient Young measure) of low-energy sequences. 

First, we recall the classical membrane theory under the positive-determinant constraint:
the energy corresponding to a deformation $\bm f:\Omega_\e\rightarrow\RR^3$ is given by
\begin{align}\label{eqn:3d-energy}
    E^{3D}_\varepsilon[\bm f] = \begin{cases}
        \frac{1}{\e}\int_{\Omega_\e}W(x',\nabla\bm f(x))\dif x &\bm f\in\mathcal{A}(\Omega_\e)\\
        +\infty &\text{otherwise},
    \end{cases}
\end{align}
where the admissible set is defined by
\begin{align*}
    \mathcal{A}(\Omega_\e):=\left\{\bm f\in W^{1,p}(\Omega_\e,\RR^3):\det\nabla\bm f>0\right\}.
\end{align*}

Using the change of variables $(\tilde x',\tilde x_3) \to (x',\frac{1}{\e} x_3)$, we obtain the rescaled energy
\begin{align*}
    E_\e[\bm f] = \begin{cases}
        \int_{\Omega_1}W(x',\nabla_\e\bm f(x))\dif x &\bm f\in\mathcal{A}_\e\\
        +\infty &\text{otherwise in } W^{1,p}(\Omega_1, \bbR^3).
    \end{cases}
\end{align*}
The admissible set for the rescaled problem $\mathcal{A}_\e$ is given by equation \eqref{eq:admiss}.

Le Dret and Raoult~\cite{le_dret_nonlinear_1995} derived the membrane $\Gamma$-limit under hypotheses excluding the positive-determinant constraint. Subsequent results of Ben Belgacem~\cite{belgacem1997methode,belgacem2000relaxation} and Anza-Hafsa and Mandallena~\cite{anza_nonlinear_2008} established the same form of the limit under this constraint. The limiting functional has the form
\begin{align*}
    E_{\text{LDR}}[\bm u]:=\int_\Omega QW_0(x',\nabla\bm u)\dif x',
\end{align*}
where $QW_0$ represents the quasiconvexification of $W_0$. 

Before stating the main result of this subsection, we need one final definition:
\begin{definition}
    For $\bm u\in W^{1,p}(\Omega,\RR^3)$, define the following admissible class of measures
\begin{align*}
    \mathcal{A}_{\bm u}:=\{\nu\in\mathcal{JY}^{p,-q}(\Omega,\RR^{3\times2}):[\nu_{x'}]=\nabla\bm u(x')\text{ a.e. in }\Omega\}
\end{align*}
and the set of minimizers
\begin{align*}
    \mathcal{M}_{\bm u}:=\arg\min_{\nu\in\mathcal{A}_{\bm u}}\mathcal{E}_0[\nu].
\end{align*}
\end{definition}

We now state the main result of this section: First, we recover the classical membrane theory $ E_{\text{LDR}}[\bm u]$ in terms of $\mathcal{E}_0[\nu]$: 
\begin{proposition}\label{prop:relation-to-LDR}
    For $\bm u\in W^{1,p}(\Omega,\RR^3)$
    \begin{align*}
         E_{\text{LDR}}[\bm u]=\min_{\nu\in\mathcal{A}_{\bm u}}\mathcal{E}_0[\nu].
    \end{align*}
    In particular, $\mathcal{M}_{\bm u}\neq\emptyset$.
\end{proposition}
As a corollary, we characterize the limiting Young measure of low-energy sequences: 
\begin{corollary}
\label{cor:linkrelax}
    Suppose $\bm u\in W^{1,p}(\Omega,\RR^3)$, identified with its trivial extension to $\Omega_1$ and let $\{\bm f_\varepsilon\}\subset\mathcal{A}_\varepsilon$ be such that $\bm f_\varepsilon\weakarrow\bm u$ in $W^{1,p}(\Omega_1,\RR^3)$ and 
    \begin{align*}
        \limsup_{\varepsilon\rightarrow0}E_\varepsilon[\bm f_\varepsilon]\leq E_{\text{LDR}}[\bm u].
    \end{align*}
    Then the family $\{\mathsf{R}\delta_{\nabla_\varepsilon\bm f_\varepsilon}\}$ is sequentially weakly${}^*$ precompact and all its limit points belong to $\mathcal{M}_{\bm u}$. Moreover, $\lim_{\varepsilon\rightarrow0}E_\varepsilon[\bm f_\varepsilon]= E_{\text{LDR}}[\bm u]$.
\end{corollary}

\section{Preliminary results on Young measures}\label{sec:notation}
We define the following class of elementary Young measures that plays an important role in the theory:
\begin{definition}[Elementary Young measure]
    Suppose $u:\Omega\rightarrow V$ is a measurable function. Then, the \textit{elementary Young measure} associated to $u$ is defined by $\mu_x\equiv\delta_{u(x)}$ for almost every $x\in\Omega$.
\end{definition}

The barycenter satisfies the following property:
\begin{lemma}\label{lem:weak-convergence-barycenter}
    Let $\{u_n\}\subset L^p(\Omega,V)$, $1<p<\infty$, be a uniformly bounded sequence generating the Young measure $\mu\in\mathcal{Y}(\Omega,V)$. Then,
    \begin{align*}
        u_n\weakarrow\innerpdt{\mu_x,\text{id}}\text{ in }L^p(\Omega,V).
    \end{align*}
\end{lemma}
\begin{proof}
    Refer to~\cite[Lemma 4.11]{rindler_calculus_2018}.
\end{proof}

We state the following fundamental property of Young measures:
\begin{theorem}\label{thm:fundamental-property-of-Young-measures}
    Let $f:\Omega\times V\rightarrow[0,+\infty]$ be a non-negative Carath\'eodory function and suppose the sequence of measurable functions $\{u^k\}$ generates the Young measure $\mu\in\mathcal{Y}(\Omega,V)$. Then,
    \begin{align}\label{eqn:Young-measure-lower-semicontinuity}
        \liminf_{k\rightarrow\infty}\int_\Omega f(x,u^k(x))\dif x\geq\int_\Omega\innerpdt{\mu_x,f(x,\cdot)}\dif x.
    \end{align}
    Moreover, the sequence of functions $x\mapsto f(x,u^k(x))$ is weakly precompact in $L^1(\Omega)$ if and only if
    \begin{align*}
        f(\cdot,u^k(\cdot))\weakarrow \bar{f}\text{ in }L^1(\Omega),
    \end{align*}
    where $\bar{f}(x):=\int_V f(x,A)\dif\mu_x(A)$.
\end{theorem}

As mentioned earlier, since we are concerned with gradient Young measures, in our applications $V=\RR^{m\times n}$ (in particular $m=3$ and $n=2$). In general, it is not true that every sequence that generates a gradient Young measure is a sequence of gradients. We recall the following celebrated result of Kinderlehrer and Pedregal~\cite{kinderlehrer_characterizations_1991}, which characterizes $p$-gradient Young measures:
\begin{theorem}[Kinderlehrer-Pedregal]\label{thm:kinderlehrer-pedregal}
    Let $p\in[1,\infty)$. A parametrized measure $\nu\in L_w^\infty(\Omega,\mathcal{M}(\RR^{m\times n}))$ is a $p$-gradient Young measure if and only if the following conditions hold:
    \begin{enumerate}
        \item[(i)] $\int_\Omega\int_{\RR^{m\times n}}\abs{\bm F}^p\dif\nu_x(\bm F)\dif x<\infty$,
        \item[(ii)] $\innerpdt{\nu_x,\text{id}}=\nabla\bm u(x)$, $\bm u\in W^{1,p}(\Omega,\RR^m)$,
        \item[(iii)] $\innerpdt{\nu_x,f}\geq f(\innerpdt{\nu_x,\text{id}})$ a.e.~$x\in\Omega$ and for all quasiconvex $f$ satisfying $\abs{f(\bm F)}\leq C(1+\abs{\bm F}^p)$.
    \end{enumerate}
\end{theorem}

We will also frequently use the following stability property of Young measure generation:
\begin{lemma}\label{lemma:strong-approximation-YM-and-equiintegrability}
    Let $1<p<\infty$ and $U\subset\RR^2$ be bounded and measurable. Let $\{\bm F^k\}$, $\{\bm G^k\}\subset L^p(U,\RR^{3\times 2})$ be norm bounded sequences. Then
    \begin{enumerate}
        \item[(i)] if $\bm F^k-\bm G^k\rightarrow0$ in measure and $\{\bm F^k\}$ generates the Young measure $\nu$, then $\{\bm G^k\}$ also generates $\nu$,
        \item[(ii)] if $\norm{\bm F^k-\bm G^k}_{L^p(U)}\rightarrow0$ and $\{\bm F^k\}$ is $p$-equiintegrable, then $\{\bm G^k\}$ is also $p$-equiintegrable.
    \end{enumerate}
\end{lemma}
\begin{proof}
    \begin{enumerate}
        \item[(i)] We test against functions $\psi(x)\phi(\bm F)$ with $\psi\in L^\infty(U)$, $\phi\in C_c(\RR^{3\times2})$. For $\eta>0$, define $U^k_\eta:=\{x\in U:\abs{\bm F^k(x)-\bm G^k(x)}>\eta\}$. On $U\setminus U^k_\eta$, we have $\abs{\phi(\bm F^k)-\phi(\bm G^k)}\leq\omega_\phi(\eta)$, where $\omega_\phi$ is the modulus of continuity of $\phi$. On $U^k_\eta$, we can bound the integrand as $\abs{\psi(\phi(\bm F^k)-\phi(\bm G^k))}\leq2\norm{\psi}_{L^\infty}\norm{\phi}_{L^\infty}$ and use $\abs{U^k_\eta}\rightarrow0$ to show that
        \begin{align*}
            \limsup_{k\rightarrow\infty}\abs{\int_U\psi(\phi(\bm F^k)-\phi(\bm G^k))\dif x}\leq\norm{\psi}_{L^1}\omega_\phi(\eta).
        \end{align*}
        Letting $\eta\rightarrow0$ shows that $\{\bm G^k\}$ generates the same Young measure as $\{\bm F^k\}$.

        \item[(ii)] From the pointwise bound $\abs{\bm G^k}^p\leq2^{p-1}\left(\abs{\bm F^k}^p+\abs{\bm F^k-\bm G^k}^p\right)$, on any measurable set $E\subset U$,
        \begin{align*}
            \int_{E}\abs{\bm G^k}^p\dif x\leq 2^{p-1}\left(\int_{E}\abs{\bm F^k}^p\dif x+\norm{\bm F^k-\bm G^k}_{L^p(E)}^p\right).
        \end{align*}
        Now, let $\varepsilon>0$, then take $N>0$ large enough so that $2^{p-1}\norm{\bm F^k-\bm G^k}_{L^p(U)}^p\leq\frac{\varepsilon}{2}$ for $k>N$. Then using the $p$-equiintegrability of $\{\bm F^k\}$, choose $\delta_0>0$ small enough so that $2^{p-1}\int_{E}\abs{\bm F^k}^p\dif x<\frac{\varepsilon}{2}$ for all $k$ when $\abs{E}<\delta_0$. For each of the remaining $N$ terms, we can find $\delta_k>0$, $k=1,...,N$ such that $\int_{E}\abs{\bm G^k}^p\dif x<\varepsilon$ for $\abs{E}<\delta_k$ (from the absolute continuity of the Lebesgue integral). Picking $\delta:=\min_{k=0,...,N}\delta_k$, the $p$-equiintegrability of $\{\bm G^k\}$ follows.
    \end{enumerate}
\end{proof}

\section{Proof of Theorem \ref{thm:gym} and Proposition \ref{prop:smooth-YM-generation}}\label{sec:young-measure-generation}

\subsection{Origami-based replacement}

As mentioned in the introduction, the origami construction of Lemma~\ref{lemma:replacement-inside-triangle} relies on the extension theorem of Brehm~\cite{brehm1981extensions}. Recall that for a closed convex polyhedral set $K\subset\RR^m$, a map $\bm{\Psi}:K\rightarrow\RR^n$ is called a \textit{piecewise affine isometry}  (or origami map) if it is continuous and there exist finitely many closed polyhedral sets $K_1,...,K_N$ with $K=\bigcup_{i=1}^NK_i$ such that $\bm{\Psi}|_{K_i}$ is an affine isometry for each $i$.
\begin{theorem}[Brehm extension]\label{thm:brehm}
    Let $m\leq n$, $M\subset\RR^m$ be a finite set and $\bm{\Psi}:M\rightarrow\RR^n$ be distance-reducing, i.e. $\abs{\bm{\Psi}(x)-\bm{\Psi}(y)}\leq\abs{x-y}$ for all $x,y\in M$. Then, there exists a piecewise affine isometry $\tilde{\bm{\Psi}}:\RR^m\rightarrow\RR^n$ with $\tilde{\bm{\Psi}}|_{M}=\bm{\Psi}$.
\end{theorem}
In the context of our proof, we apply this extension with $m=2$ and $n=3$ and the finitely many polyhedral pieces can be assumed to be a finite triangulation of the domain. Furthermore, since each piece of $\tilde{\bm{\Psi}}$ is an affine isometry, the map itself is distance-reducing, i.e., $1$-Lipschitz and $\nabla\tilde{\bm{\Psi}}$ has orthonormal columns, thus $\abs{\partial_1\tilde{\bm{\Psi}}\times\partial_2\tilde{\bm{\Psi}}}=1$ a.e.

With Theorem~\ref{thm:brehm} in hand, we are ready to present Lemma~\ref{lemma:replacement-inside-triangle}, which allows us to replace a nondegenerate, strictly short affine map on a triangle $T$ by a piecewise affine map with the same boundary values whose Jacobian equals one away from a thin boundary layer. The construction is illustrated in Figure~\ref{fig4}, and the proof proceeds in two steps. In Step~1, we work on a shrunk copy of $T$ and superimpose a zigzag profile on the affine boundary data. The zigzag is chosen so that the resulting map is a piecewise isometry on each side of the shrunk triangle, and so that its restriction to the finite set of all peak and valley points is distance reducing. We then apply Theorem~\ref{thm:brehm} to extend this map to a piecewise affine isometry on the shrunk triangle. In Step~2, we interpolate between this map on the boundary of the shrunk triangle and the original affine map on $\partial T$ by an explicit piecewise affine construction in the collar region between the two triangles. 

\begin{figure}[!t]
\centering
\input{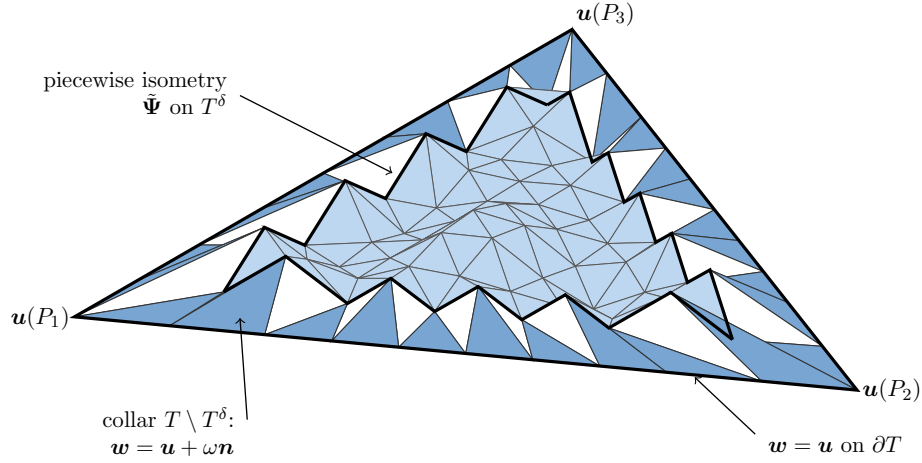}
\caption{A representation of the map $\bm w$ constructed in Lemma~\ref{lemma:replacement-inside-triangle}. On the boundary of the inner triangle, $\partial T^\delta$, $\bm w$ restricts to the zigzag function $\bm \Psi$. In the interior, $T^\delta$, $\bm{w}=\tilde{\bm{\Psi}}$ is given by a piecewise isometry constructed using Brehm's extension theorem. In the collar region $T\setminus T^\delta$, $\bm{w}$ is a piecewise linear interpolation between the zigzag map on $\partial T^\delta$ and the affine boundary map on $\partial T$.}
\label{fig4}
\end{figure}

We use $D^\lambda(x):=\lambda x$ to denote a dilation in $\RR^2$ by a factor $\lambda$.
\begin{lemma}\label{lemma:replacement-inside-triangle}
    Let $T\subset\RR^2$ be a nondegenerate triangle whose interior angles are bounded below by $a_0>0$ (without loss of generality, assume that the centroid of $T$ lies at the origin) and let $\bm u:T\rightarrow\RR^3$ be a non-singular, strictly short affine map, i.e., $\bm u=\bm Ax+\bm b$ for some $\bm A\in \RR^{3\times 2}$ and $\bm b\in\RR^3$, $\abs{\bm u(x)-\bm u(y)}\leq\abs{x-y}$, $J(\bm A)\neq0$ and the largest singular value $\sigma_1(\bm A)<1$. For $0<\delta<1$, let $T^\delta:=D^{1-\delta} T$. Then, there exists a continuous and piecewise affine map $\bm w:T\rightarrow\RR^3$ such that $\bm w=\bm u$ on $\partial T$ satisfying
    \begin{alignat*}{2}
        \mathrm{Lip}(\bm w) &\leq C_1, \\
        \abs{\partial_1 \bm w \times \partial_2 \bm w}&\geq \abs{\partial_1 \bm u \times \partial_2 \bm u}&&\quad \text{on } T, \\
        \abs{\partial_1 \bm w \times \partial_2 \bm w}&= 1&&\quad \text{on } T^\delta,
    \end{alignat*}
    for some constant $C_1\equiv C_1(a_0)>0$.
\end{lemma}
\begin{proof}
    \textbf{Step 1: Construction in the shrunk triangle.} Let $T$ be the triangle defined by points $P_1,P_2,P_3$ and $T^\delta$ be the shrunk triangle defined by points $P_1', P_2', P_3'$. 
    We will make reference to Figure~\ref{fig1} for the nomenclature. Index $i$ takes values in $\{1,2,3\}$ and is meant $\mathrm{mod} \, 3$. Denote the segment $\ell_i:=P_iP_{i+1}$ and $\ell_i':=P_i'P_{i+1}'$.
    
    Let $L_i=\abs{P_{i+1}-P_i}$, from which, obviously, $\abs{P_{i+1}'-P_i'}=(1-\delta)L_i=:L_i'$. Also set $a_i:=\frac{P_{i+1}-P_i}{L_i} = \frac{P_{i+1}'-P_i'}{L_i'}$.

    Consider the restriction of $\bm u$ along the segments $\ell_i'$ given by:
    \begin{align*}
        \bm \gamma_i(t)&=\bm u(P_i')+t\bm \alpha_i\qquad \text{ for }t\in[0,L_i'],
    \end{align*}
    where $\bm\alpha_i= \bm A a_i$. Since $\bm {u}$ is a non-singular strictly short affine map, the image of $T$ under $\bm u$ is contained in a plane with unit normal  $\bm n=\frac{\bm \alpha_1\times\bm \alpha_2}{\abs{\bm \alpha_1\times\bm \alpha_2}}$.  Clearly $\bm n\perp\bm\alpha_i$ and since $\bm {u}$ is short, $\abs{\bm\alpha_i}\leq 1$. 

    Let $\bm B:=\bm I-\bm A^T\bm A$ which is positive definite since $\bm A$ is short. Denote the smallest eigenvalue of $\bm B$ by
    \begin{align*}
        \lambda:=1-\sigma_1^2>0,
    \end{align*}
    where $\sigma_1$ is the largest singular value of $\bm A$. For $k=1,2,3$ denote $\beta_k=\sqrt{1-\abs{\bm\alpha_k}^2}$, which satisfies $\sqrt{\lambda}\leq\beta_k\leq1$. Let $\psi(t):=1-\abs{1-(t\mod 2)}$ be the 2-periodic tent function with $\psi\in[0,1]$, $\abs{\psi'}=1$ a.e., and $\psi(2\ZZ)=0$. Let $3b_k>0$ denote the height of the triangle $T$ with respect to the side $\ell_k$. Set
    \begin{align*}
        m:=\left\lceil\frac{1}{2\sqrt{\lambda}\sin a_0}\right\rceil
    \end{align*}
    and
    \begin{align}\label{eq:choice-of-H}
        0<H\leq\min\left\{\frac{\min_k(\beta_kL_k')}{8m},\delta\min_{k}b_k\right\}.
    \end{align}

    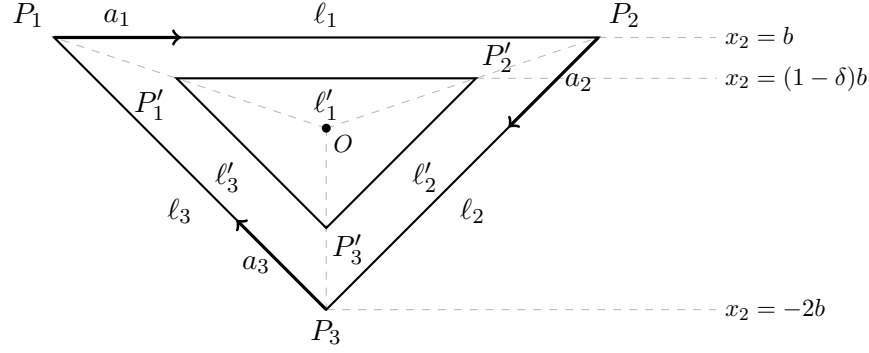
\begin{figure}[!t]
    \centering
    \begin{tikzpicture}[scale=1.2]
% Centroid at the origin; P_1P_2 on x_2 = b, P_3 on x_2 = -2b  (here b = 1)
\coordinate (P) at (-3,1);
\coordinate (Q) at (3,1);
\coordinate (R) at (0,-2);
% Dilate about the centroid, ratio 1-delta = 0.55
\coordinate (PP) at (-1.65,0.55);
\coordinate (QQ) at (1.65,0.55);
\coordinate (RR) at (0,-1.1);
\coordinate (C) at (0,0);
% Level lines (drawn only to the right of the triangle, as leaders to the labels)
\draw[gray!55,dashed] (3,1) -- (4.3,1) node[right,black,scale=0.85] {$x_2=b$};
\draw[gray!55,dashed] (1.65,0.55) -- (4.3,0.55) node[right,black,scale=0.85] {$x_2=(1-\delta)b$};
\draw[gray!55,dashed] (0,-2) -- (4.3,-2) node[right,black,scale=0.85] {$x_2=-2b$};
% Dilation rays
\draw[gray!55,dashed] (C) -- (P);
\draw[gray!55,dashed] (C) -- (Q);
\draw[gray!55,dashed] (C) -- (R);
% Large triangle
\draw[thick] (P) -- (Q) -- (R) -- cycle;
% Smaller triangle
\draw[thick] (PP) -- (QQ) -- (RR) -- cycle;
% Centroid
\fill (C) circle (1.4pt);
\node[below right, inner sep=3pt, scale=0.85] at (C) {$O$};
% Direction arrows
\draw[->, very thick, black] (P) -- ($(P)!1.4cm!(Q)$)
    node[midway, above] {$a_1$};
\draw[->, very thick, black] (Q) -- ($(Q)!1.4cm!(R)$)
    node[midway, right] {$a_2$};
\draw[->, very thick, black] (R) -- ($(R)!1.4cm!(P)$)
    node[midway, left] {$a_3$};
% Vertex labels, outer
\node[above left] at (P) {$P_1$};
\node[above right] at (Q) {$P_2$};
\node[below] at (R) {$P_3$};
% Vertex labels, inner
\node[below left, inner sep=4pt] at (PP) {$P_1'$};
\node[above right, inner sep=2pt] at (QQ) {$P_2'$};
\node[below right, inner sep=2pt] at (RR) {$P_3'$};
% Edge labels, outer
\node[above] at ($(P)!0.5!(Q)$) {$\ell_1$};
\node[xshift=9pt, yshift=-9pt] at ($(Q)!0.55!(R)$) {$\ell_2$};
\node[xshift=-9pt, yshift=-9pt] at ($(R)!0.45!(P)$) {$\ell_3$};
% Edge labels, inner
\node[below, inner sep=4pt] at ($(PP)!0.5!(QQ)$) {$\ell_1'$};
\node[xshift=9pt, yshift=-9pt] at ($(QQ)!0.5!(RR)$) {$\ell_2'$};
\node[xshift=-9pt, yshift=-9pt] at ($(RR)!0.5!(PP)$) {$\ell_3'$};
\end{tikzpicture}
    \caption{Nomenclature used in Lemma \ref{lemma:replacement-inside-triangle}. The domain is a triangle $T=\triangle P_1P_2P_3$ with centroid located at the origin $O$. The triangle $T^\delta=\triangle P_1'P_2'P_3'$ is a $(1-\delta)$-dilated copy of $T$. The edges $P_iP_{i+1}$ and $P_i'P_{i+1}'$ are denoted $\ell_i$ and $\ell_{i}'$ respectively, while the unit vectors along the edges are denoted $a_i$. For the latter part of the proof of Lemma~\ref{lemma:replacement-inside-triangle}, we reorient the triangle so that the edge $P_1P_2$ coincides with $x_2=b$ for some $b$ and consequently, the height of $T$ with respect to the base $P_1P_2$ is $3b$.}
    \label{fig1}
\end{figure}
    
    Let $q_k:=\frac{2mH}{\beta_k}$, then divide each segment $\ell_k'$ into three segments along its length: $[0,q_k]$, $[q_k,L_k'-q_k]$ and $[L_k'-q_k,L_k']$. Since $q_k\leq\frac{L_k'}{4}$, the middle segment has length $r_k:=L_k'-2q_k\geq\frac{L_k'}{2}$. We will introduce out-of-plane oscillations on every side of the triangle in such a way that the restriction to a given side is unit-speed and piecewise affine. We choose the number of teeth and amplitudes to be equal to $m$ and $H$ respectively on the two outer segments of every side. The amplitudes and frequencies on the inner segments are chosen so that the distance shortening property is preserved on the collection of peak/valley points on every edge. The reason for this specific choice is that near a common vertex, points get arbitrarily close and the proof of shortness of the map is simplified by choosing the oscillations to be synchronized. Away from the vertices, the reference distance between edges is bounded below (since the triangle angle is bounded below) so there is more flexibility in choosing the out-of-plane oscillations. Define
    \begin{align*}
        W(t):=\begin{cases}
            H\psi(\beta_kt/H), &0\leq t\leq q_k\\
            h_k\psi(\beta_k(t-q_k)/h_k) &q_k< t\leq L'_k-q_k\\
            H\psi(\beta_k(L_k'-t)/H) &L_k'-q_k<t\leq L_k'
        \end{cases}
    \end{align*}
    and define along side-$k$, $\bm{\Psi}_k(t)=\bm \gamma_k(t)+W(t)\bm n$. Here, we choose an integer $N_k\geq\frac{\beta_kr_k}{2H}$ and set $h_k:=\frac{\beta_k r_k}{2N_k}\leq H$. It is straightforward to see that $\abs{\bm{\Psi}_k'}=1$ a.e. and is a piecewise isometry on each side (refer to Figure~\ref{fig2}). 
    
    Let $\bm{\Psi}$ denote the full boundary map, whose restriction to $\ell_k'$ is $\bm{\Psi}_k$. Define the quantity
    \begin{align*}
        D(x,y):=\abs{x-y}^2-\abs{\bm{\Psi}(x)-\bm{\Psi}(y)}^2,
    \end{align*}
    for $x,y\in\partial T^\delta$. We want to verify the length-shortening property for the peak and valley points on the edges, i.e., $D(x,y)\geq 0$. Since the map restricted to a single edge is length preserving, we only need to verify the condition for points that are on different edges. Let $j=i+1$ and $P'=P'_{j}$ be the common vertex of $\ell_i'$, $\ell_j'$ and write
    \begin{align*}
        x=P'-ua_i\quad y=P'+sa_j,
    \end{align*}
    where $-a_i$ and $a_j$ are unit vectors pointing along $\ell_i'$ and $\ell_j'$ from $P'$, and $u,s\geq0$ are the distances of the two points from $P'$. Let $W_i,W_j$ denote the values of the out-of-plane displacement $W$ along these two segments. Then
    \begin{align*}
        D(x,y)\equiv D(u,s)=(ua_i+sa_j)^T\bm B(ua_i+sa_j)-(W_i-W_j)^2
    \end{align*}
    since $\bm n\perp \bm A(x-y)$. Simplifying, we get
    \begin{align*}
        D(u,s)=u^2\beta_i^2-2us\beta_i\beta_jc+s^2\beta_j^2-(W_i-W_j)^2,
    \end{align*}
    where, by the Cauchy--Schwarz inequality,
    \begin{align*}
        c:=-\frac{a_i^T\bm B a_j}{\beta_i\beta_j}\leq 1.
    \end{align*}

    We now split into two cases:
    \begin{enumerate}
        \item Case 1. $u\leq q_i$ and $s\leq q_j$: All the peaks have the same height. Let $p,q\in\ZZ$ and consider $u_p$, $s_q$ such that
        \begin{align*}
            \frac{u_p\beta_i}{H}=p\quad\text{and}\quad\frac{s_q\beta_j}{H}=q.
        \end{align*}
        Peaks ($\psi=1$) correspond to odd values of $p$ and $q$ while valleys ($\psi=0$) correspond to even values, so that $W_i=H\varepsilon_p$ and $W_j=H\varepsilon_q$ with $\varepsilon_p:=p\bmod2$. Note $(\varepsilon_p-\varepsilon_q)^2=((p-q)\bmod 2)^2$, so
        \begin{align*}
            D(u_p,s_q)&=H^2\left[p^2-2pqc+q^2-(\varepsilon_p-\varepsilon_q)^2\right]\\
            &=H^2\left[(p-q)^2+2pq(1-c)-(\varepsilon_p-\varepsilon_q)^2\right]
        \end{align*}
        For the peak-peak or valley-valley case ($p$, $q$ both even or both odd):
        \begin{align*}
            D(u_p,s_q)=H^2\left[(p-q)^2+2pq(1-c)\right]\geq0,
        \end{align*}
        since $p,q\geq0$ and $c\leq 1$. For the peak-valley case, i.e., when one of $p$ and $q$ is even and the other is odd, so $\abs{p-q}\geq1$, and
        \begin{align*}
            D(u_p,s_q)=H^2\left[(p-q)^2+2pq(1-c)-1\right]\geq0,
        \end{align*}
        since $\abs{p-q}\geq 1$, $p,q\geq0$ and $c\leq1$.

        \item Case 2. Either $u>q_i$ or $s>q_j$: Without loss of generality, assume $u>q_i=\frac{2mH}{\beta_i}\geq2mH$, using $\beta_i\leq1$. Since $x=P'-ua_i\in\ell_i'$  and $y=P'+sa_j\in\ell_j'$, 
        \begin{align}\label{eq:H-sqrt-lambda}
            \abs{x-y}\geq u\sin a_0\geq2mH\sin a_0\geq\frac{H}{\sqrt{\lambda}}.
        \end{align}
        Then computing the distance,
        \begin{align*}
            \abs{\bm{\Psi}(x)-\bm{\Psi}(y)}^2&=\abs{\bm A(x-y)}^2+(W_i-W_j)^2\\
            &\leq (1-\lambda)\abs{x-y}^2+H^2\leq\abs{x-y}^2,
        \end{align*}
        since $W\leq H$ and $H^2\leq\lambda\abs{x-y}^2$ from~\eqref{eq:H-sqrt-lambda}.
    \end{enumerate}
    
    We apply Theorem~\ref{thm:brehm} to the collection of peak and valley points to get a piecewise isometry $\tilde{\bm{\Psi}}$ on $T^\delta$ that agrees with the sample points on the boundary. In particular, $\tilde{\bm{\Psi}}$ agrees with $\bm{\Psi}$ on $\partial T^\delta$ because $\bm\Psi$ is affine and distance preserving between the sample points. Since $\tilde{\bm{\Psi}}$ is a piecewise isometry, we get the desired lower bound on $\abs{\partial_1\tilde{\bm{\Psi}}\times\partial_2\tilde{\bm{\Psi}}}$.

    \textbf{Step 2: Interpolation in the collar region.} In the remaining set $T\setminus T^\delta$, we extend $\bm w$ as a piecewise affine map in the form
    \begin{align*}
        \bm w = \bm u+\omega\bm n,
    \end{align*}
    where $\omega:T\setminus T^\delta\rightarrow\RR$ is continuous and piecewise affine. On any affine piece of $\bm w$, we have $\nabla\bm w = \bm A+\bm n\otimes\nabla\omega$ so 
    \begin{align*}
        \partial_1\bm w\times\partial_2\bm w = \bm Ae_1\times\bm A e_2+\omega_{,2}\bm Ae_1\times\bm n+\omega_{,1}\bm n\times \bm Ae_2.
    \end{align*}
    Since $\bm n$ is orthogonal to the image plane of $T$ under $\bm A$, the first term above is parallel to $\bm n$, while the latter two terms are orthogonal to $\bm n$. Projecting the vector on $\bm n$,
    \begin{align*}
        \abs{\partial_1\bm w\times\partial_2\bm w}\geq\abs{\bm Ae_1\times\bm A e_2}=\abs{\partial_1\bm u\times\partial_2\bm u}\quad\text{on }T\setminus T^\delta.
    \end{align*}
    Also from the same orthogonality, for any vector $v\in\RR^2$, $\abs{\nabla \bm w\ v}^2=\abs{\bm Av}^2+(\nabla\omega\cdot v)^2$, and so using the shortness of $\bm A$,
    \begin{align}\label{eq:w-lipschitz-bound-in-terms-of-omega}
        \abs{\nabla\bm w}\leq\sqrt{\sigma_1^2+\abs{\nabla\omega}^2}\leq\sqrt{1+\abs{\nabla\omega}^2}\text{ on }T\setminus T^\delta.
    \end{align}

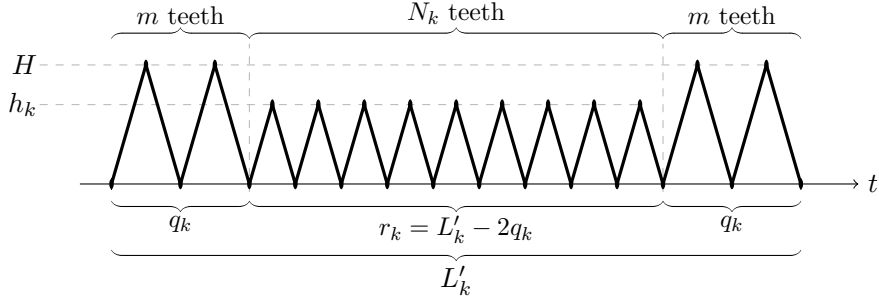
\begin{figure}[!t]
    \centering
    \begin{tikzpicture}[xscale=0.76,yscale=1.75]
    \def\Hh{0.9}\def\hk{0.6}\def\q{2.4}\def\Lk{12}\def\Rr{9.6}
     
    % --- amplitude guides, drawn out to the left for the labels
    \draw[gray!55,dashed] (-1.25,\Hh) -- (11.4,\Hh);
    \draw[gray!55,dashed] (-1.25,\hk) -- (9.2,\hk);
    \node[left,inner sep=2pt,scale=0.95] at (-1.2,\Hh) {$H$};
    \node[left,inner sep=2pt,scale=0.95] at (-1.2,\hk) {$h_k$};
     
    % --- region separators
    \draw[gray!55,dashed] (\q,-0.06) -- (\q,1.06);
    \draw[gray!55,dashed] (\Rr,-0.06) -- (\Rr,1.06);
     
    % --- axis and end ticks
    \draw[->] (-0.55,0) -- (13.0,0) node[right,scale=0.95] {$t$};
    \draw (0,-0.045) -- (0,0.045);
    \draw (\Lk,-0.045) -- (\Lk,0.045);
     
    % --- the profile
    \draw[very thick] (0.0,0) -- (0.6,0.9) -- (1.2,0) -- (1.8,0.9) -- (2.4,0) -- (2.8,0.6) -- (3.2,0) -- (3.6,0.6) -- (4.0,0) -- (4.4,0.6) -- (4.8,0) -- (5.2,0.6) -- (5.6,0) -- (6.0,0.6) -- (6.4,0) -- (6.8,0.6) -- (7.2,0) -- (7.6,0.6) -- (8.0,0) -- (8.4,0.6) -- (8.8,0) -- (9.2,0.6) -- (9.6,0) -- (10.2,0.9) -- (10.8,0) -- (11.4,0.9) -- (12.0,0);
    \fill (0.0,0) circle (1.1pt); \fill (0.6,0.9) circle (1.1pt); \fill (1.2,0) circle (1.1pt); \fill (1.8,0.9) circle (1.1pt); \fill (2.4,0) circle (1.1pt); \fill (2.8,0.6) circle (1.1pt); \fill (3.2,0) circle (1.1pt); \fill (3.6,0.6) circle (1.1pt); \fill (4.0,0) circle (1.1pt); \fill (4.4,0.6) circle (1.1pt); \fill (4.8,0) circle (1.1pt); \fill (5.2,0.6) circle (1.1pt); \fill (5.6,0) circle (1.1pt); \fill (6.0,0.6) circle (1.1pt); \fill (6.4,0) circle (1.1pt); \fill (6.8,0.6) circle (1.1pt); \fill (7.2,0) circle (1.1pt); \fill (7.6,0.6) circle (1.1pt); \fill (8.0,0) circle (1.1pt); \fill (8.4,0.6) circle (1.1pt); \fill (8.8,0) circle (1.1pt); \fill (9.2,0.6) circle (1.1pt); \fill (9.6,0) circle (1.1pt); \fill (10.2,0.9) circle (1.1pt); \fill (10.8,0) circle (1.1pt); \fill (11.4,0.9) circle (1.1pt); \fill (12.0,0) circle (1.1pt);
     
    % --- tooth counts above
    \draw[decorate,decoration={calligraphic brace,amplitude=4pt},gray!60!black,thin]
      (0,1.10) -- (\q,1.10);
    \node[above,scale=0.88,inner sep=6pt] at (1.2,1.10) {$m$ teeth};
    \draw[decorate,decoration={calligraphic brace,amplitude=4pt},gray!60!black,thin]
      (\q,1.10) -- (\Rr,1.10);
    \node[above,scale=0.88,inner sep=6pt] at (6,1.10) {$N_k$ teeth};
    \draw[decorate,decoration={calligraphic brace,amplitude=4pt},gray!60!black,thin]
      (\Rr,1.10) -- (\Lk,1.10);
    \node[above,scale=0.88,inner sep=6pt] at (10.8,1.10) {$m$ teeth};
     
    % --- lengths below
    \draw[decorate,decoration={calligraphic brace,amplitude=4pt,mirror},gray!60!black,thin]
      (0,-0.13) -- (\q,-0.13);
    \node[below,scale=0.88,inner sep=6pt] at (1.2,-0.13) {$q_k$};
    \draw[decorate,decoration={calligraphic brace,amplitude=4pt,mirror},gray!60!black,thin]
      (\q,-0.13) -- (\Rr,-0.13);
    \node[below,scale=0.88,inner sep=6pt] at (6,-0.13) {$r_k=L_k'-2q_k$};
    \draw[decorate,decoration={calligraphic brace,amplitude=4pt,mirror},gray!60!black,thin]
      (\Rr,-0.13) -- (\Lk,-0.13);
    \node[below,scale=0.88,inner sep=6pt] at (10.8,-0.13) {$q_k$};
    \draw[decorate,decoration={calligraphic brace,amplitude=4pt,mirror},gray!60!black,thin]
      (0,-0.50) -- (\Lk,-0.50);
    \node[below,scale=0.95,inner sep=6pt] at (6,-0.50) {$L_k'$};
     
    %\node[scale=0.95] at (-1.25,1.05) {$W_k$};
    \end{tikzpicture}
    \caption{A depiction of the zigzag profile $W(t)$ along one of the inner triangle boundaries. On the two outer segments of length $q_k$, there are $m$ teeth and the amplitude is $H$, whereas on the middle segment, the amplitude is $h_k$ and there are $N_k$ teeth.}
    \label{fig2}
    \end{figure}
    
    Therefore, it suffices to construct the scalar $\omega$ with the right boundary conditions and bounds. 

    Focus on a single trapezium $P_1P_2P_2'P_1'$ (the construction works in the same way for the other two trapezia). We reorient coordinates while keeping the centroid of $T$ at the origin so that $\ell_1=P_1P_2$ lies on the line $x_2=b$ with $b=b_1>0$ (where recall $3b_k$ is the height of the triangle with respect to the base $\ell_k$) and $a_1=e_1$, and write $P_1=(\xi,b)$. Then $\ell_1'=P_1'P_2'$ lies on the line $x_2=(1-\delta)b$ and $P_i'=(1-\delta)P_i$. Since the centroid is at the origin, $P_3$ lies on $x_2=-2b$, so the height of $T$ over the base $\ell_1$ is $3b$ (refer to Figure~\ref{fig1}). Recall that we denote the length of segment $\ell_k$ by $L_k$. We then have
\begin{align}\label{eq:bound-on-distance-of-P-from-origin}
        \frac{P_1+P_2+P_3}{3}=0\implies\abs{P_1}\leq\frac{L_1+L_3}{3}.
    \end{align}
    Since the interior angles are bounded below by $a_0$, the lengths of the sides $\ell_3$ and $\ell_2$ can be estimated as
    \begin{align*}
        L_2,L_3\leq\frac{3b}{\sin a_0},
    \end{align*}
    and from the triangle inequality, $L_1\leq\frac{6b}{\sin a_0}$. Putting this together with~\eqref{eq:bound-on-distance-of-P-from-origin}, we get
    \begin{align*}
        \abs{P_1}\leq\frac{3b}{\sin a_0}.
    \end{align*}
    We also get the same bounds for $\abs{P_2}$, and thus for any point $x\in \ell_1$, we get 
    \begin{align}\label{eq:x_1-coordinate-bound}
        \abs{x_1}\leq\frac{3b}{\sin a_0}.
    \end{align}
    Let $0=t_0<t_1<...<t_N=L_1'$ be the peak/valley parameters of $W_1$ defined on $\ell_1'$ and set
    \begin{align*}
        \bar{W}_j:=W_1(t_j)\in[0,H],\quad \Delta t_j:=t_{j+1}-t_j.
    \end{align*}
    Define the inner and outer nodes
    \begin{align*}
        S_j:=P_1'+t_ja_1\in\ell_1'\qquad R_j:=\frac{1}{1-\delta}S_j=P_1+\frac{t_j}{1-\delta}a_1\in\ell_1,
    \end{align*}
    so that $R_0=P_1$, $R_N=P_2$, $S_0=P_1'$ and $S_N=P_2'$. In coordinates, write
    \begin{align*}
        R_j=(\xi_j,b),\quad S_j=(1-\delta)(\xi_j,b),\quad\xi_j:=\xi+\frac{t_j}{1-\delta}.
    \end{align*}
    Since $R_j\in P_1P_2$, the bound~\eqref{eq:x_1-coordinate-bound} applies and
    \begin{align*}
        \abs{\xi_j}\leq\frac{3b}{\sin a_0},\quad j=0,...,N.
    \end{align*}
    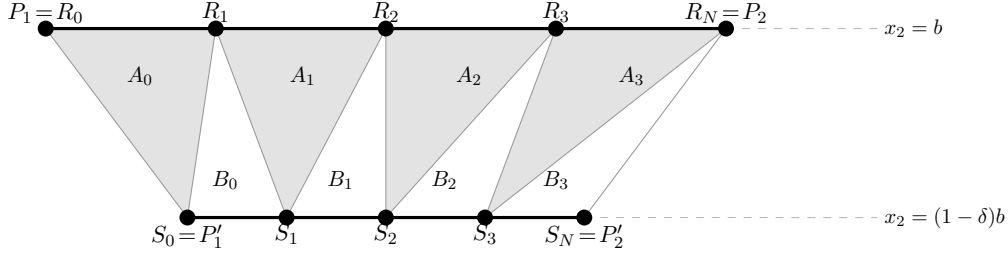
\begin{figure}[!t]
    \centering
        \begin{tikzpicture}[scale=2.5]
        % Centroid at the origin.  Outer edge on x_2 = b (b = 2.4), inner on
        % x_2 = (1-delta)b with 1-delta = 7/12, so S_j = (1-delta) R_j exactly.
        \def\yo{2.4}\def\yi{1.4}\def\rat{0.58333}
        \coordinate (C) at (0,0);
         
        % ---- shading: A_j grey, B_j white ---------------------------------------
        \fill[gray!22] (-1.8,\yo) -- (-0.9,\yo) -- (-1.05,\yi) -- cycle;   % A_0
        \fill[gray!22] (-0.9,\yo) -- (0,\yo) -- (-0.525,\yi) -- cycle;     % A_1
        \fill[gray!22] (0,\yo) -- (0.9,\yo) -- (0,\yi) -- cycle;           % A_2
        \fill[gray!22] (0.9,\yo) -- (1.8,\yo) -- (0.525,\yi) -- cycle;     % A_3
         
        % ---- dilation rays -------------------------------------------------------

        % ---- mesh edges ----------------------------------------------------------
        \foreach \x in {-1.8,-0.9,0,0.9,1.8}{
          \draw[gray!80] (\x,\yo) -- ({\rat*\x},\yi);}
        \draw[gray!80] (-1.05,\yi) -- (-0.9,\yo);
        \draw[gray!80] (-0.525,\yi) -- (0,\yo);
        \draw[gray!80] (0,\yi) -- (0.9,\yo);
        \draw[gray!80] (0.525,\yi) -- (1.8,\yo);
         
        % ---- boundary edges ------------------------------------------------------
        \draw[very thick] (-1.8,\yo) -- (1.8,\yo);
        \draw[very thick] (-1.05,\yi) -- (1.05,\yi);
        \draw[gray!55,dashed] (1.8,\yo) -- (2.6,\yo) node[right,black,scale=0.72] {$x_2=b$};
        \draw[gray!55,dashed] (1.05,\yi) -- (2.6,\yi) node[right,black,scale=0.72] {$x_2=(1-\delta)b$};
         
        % ---- nodes ---------------------------------------------------------------
        \foreach \x in {-1.8,-0.9,0,0.9,1.8}{\fill (\x,\yo) circle (1.2pt);}
        \foreach \x in {-1.05,-0.525,0,0.525,1.05}{\fill (\x,\yi) circle (1.2pt);}
         
        \node[above,inner sep=2.5pt,scale=0.8] at (-1.8,\yo) {$P_1\!=\!R_0$};
        \node[below,inner sep=3pt,scale=0.8] at (-0.525,\yi) {$S_1$};
        \node[below,inner sep=3pt,scale=0.8] at (0,\yi) {$S_2$};
        \node[below,inner sep=3pt,scale=0.8] at (0.525,\yi) {$S_3$};
         
        \node[above,inner sep=3pt,scale=0.8] at (-0.9,\yo) {$R_1$};
         
        \node[above,inner sep=3pt,scale=0.8] at (0,\yo) {$R_2$};
        \node[above,inner sep=3pt,scale=0.8] at (0.9,\yo) {$R_3$};
        \node[above,inner sep=2.5pt,scale=0.8] at (1.8,\yo) {$R_N\!=\!P_2$};
        \node[below,inner sep=3pt,scale=0.8] at (-1.05,\yi) {$S_0\!=\!P_1'$};
        \node[below,inner sep=3pt,scale=0.8] at (1.05,\yi) {$S_N\!=\!P_2'$};
         
        % ---- triangle labels -----------------------------------------------------
        \node[scale=0.75] at (-1.30,2.15) {$A_0$};
        \node[scale=0.75] at (-0.44,2.15) {$A_1$};
        \node[scale=0.75] at (0.44,2.15)  {$A_2$};
        \node[scale=0.75] at (1.30,2.15)  {$A_3$};
        \node[scale=0.75] at (-0.85,1.60) {$B_0$};
        \node[scale=0.75] at (-0.24,1.60) {$B_1$};
        \node[scale=0.75] at (0.31,1.60)  {$B_2$};
        \node[scale=0.75] at (0.90,1.60)  {$B_3$};

        \end{tikzpicture}
        \caption{Triangulation of the trapezium $P_1P_2P_2'P_1'$. We have chosen coordinates so that $\ell_1=P_1P_2$ lies along $x_2=b$ and $\ell_1'=P_1'P_2'$ lies along $x_2=(1-\delta)b$. The inner nodes, $S_0,...,S_N$ lie on the peak-valley points of the zigzag function constructed on the inner boundary $\partial T^\delta$, while the outer points $R_0,...,R_N$ are the corresponding dilates by the factor $\frac{1}{1-\delta}$ which lie on the outer boundary $\partial T$.}
        \label{fig3}
    \end{figure}
    We now construct $\omega$ as a piecewise affine function on a triangulation of $T\setminus T^\delta$: For $j=0,...,N-1$, the quadrilateral $R_jR_{j+1}S_{j+1}S_j$ is a trapezium with parallel sides lying on the lines $x_2=b$ and $x_2=(1-\delta)b$. Split this trapezium into two triangles (refer to Figure~\ref{fig3}):
    \begin{align*}
        A_j:=\triangle R_jR_{j+1}S_j\qquad B_j:=\triangle S_jS_{j+1}R_{j+1}.
    \end{align*}
    Performing the same construction on the other two trapezia, we get a triangulation of $T\setminus T^\delta$. Define $\omega$ as the unique continuous function that is affine on each triangle $A_j$, $B_j$ and satisfies
    \begin{align*}
        \omega(R_j)=0\qquad\omega(S_j)=\bar{W}_j.
    \end{align*}
    Since $R_jR_{j+1}\subset\partial T$ and $\omega(R_j)=\omega(R_{j+1})=0$, $\omega=0$ on $\partial T$ and therefore, $\bm w =\bm u$ on $\partial T$. On $S_jS_{j+1}\subset \partial T^\delta$, $\omega$ is affine with endpoint values $\bar{W}_j$ and $\bar{W}_{j+1}$, hence $\omega=W_1$ on $\ell_1'$ (and consequently the entire inner boundary $\partial T^\delta$). Thus $\bm w=\bm u+W_1\bm n=\tilde{\bm{\Psi}}$ on $\partial T^\delta$.

    We can now estimate the derivative of $\omega$ and its Jacobian. On $A_j$: since $R_{j+1}-R_{j}=(\xi_{j+1}-\xi_{j})e_1$ with $\omega(R_j)=\omega(R_{j+1})=0$, we have $\partial_1\omega = 0$. On the other hand, $R_j-S_j=\delta(\xi_j,b)$ and $\omega(R_j)-\omega(S_j)=-\bar{W}_j$, therefore $\partial_2\omega=\frac{-\bar{W}_j}{\delta b}$. Putting this together,
    \begin{align}\label{eq:grad-omega-A-formula}
        \nabla\omega|_{A_j}=\left(0,-\frac{\bar{W}_j}{\delta b}\right)
    \end{align}

    On $B_j$: $S_{j+1}-S_j=(1-\delta)(\xi_{j+1}-\xi_j)e_1=\Delta t_je_1$ and $\omega(S_{j+1})-\omega(S_{j})=\bar{W}_{j+1}-\bar{W}_{j}$, therefore, $\partial_1\omega=\pm\beta_1$ (slope of the zigzag function $W(t)$). In the remaining direction, $R_{j+1}-S_{j+1}=\delta(\xi_{j+1},b)$ and $\omega(R_{j+1})-\omega(S_{j+1})=-\bar{W}_{j+1}$. Then using the directional derivative formula
    \begin{align*}
        \nabla\omega|_{B_j}\cdot\delta(\xi_{j+1},b)=\delta\partial_1\omega\xi_{j+1}+\delta\partial_2\omega b=-\bar{W}_{j+1}.
    \end{align*}
    Solving for $\partial_2\omega$, we get $\partial_2\omega=-\frac{\bar{W}_{j+1}}{\delta b}\mp\frac{\beta_1\xi_{j+1}}{b}$, thus 
    \begin{align}\label{eq:grad-omega-B-formula}
        \nabla\omega|_{B_j}=\left(\pm\beta_1,-\frac{\bar{W}_{j+1}}{\delta b}\mp\frac{\beta_1\xi_{j+1}}{b}\right).
    \end{align}
    Using $0\leq \bar{W}_j\leq H$, $\beta_1\leq 1$,~\eqref{eq:x_1-coordinate-bound},~\eqref{eq:grad-omega-A-formula} and~\eqref{eq:grad-omega-B-formula}, we get
    \begin{align*}
        \abs{\partial_1\omega}\leq 1,\qquad\abs{\partial_2\omega}\leq\frac{H}{\delta b}+\frac{3}{\sin a_0}.
    \end{align*}
    From the choice of $H$ in~\eqref{eq:choice-of-H}, we have $\frac{H}{\delta b}\leq 1$,
    \begin{align*}
        \abs{\nabla\omega}\leq\sqrt{1+\left(1+\frac{3}{\sin a_0}\right)^2}\qquad T\setminus T^\delta.
    \end{align*}
    Combining with~\eqref{eq:w-lipschitz-bound-in-terms-of-omega}, 
    \begin{align*}
        \mathrm{Lip}(\bm w)\leq\sqrt{2+\left(1+\frac{3}{\sin a_0}\right)^2}=:C_1(a_0).
    \end{align*}
\end{proof}

\subsection{Piecewise affine generation}

A key step in the proof of our Young measure generation theorem is the approximation of generating maps by immersions. The following result of Gromov and Eliashberg~\cite{gromov_construction_1971} says that immersions are $W^{1,p}$-dense in the class of $C^1$ maps from $\RR^2$ into $\RR^3$. We state the version of this result presented in~\cite[Theorem A.4]{anza_nonlinear_2008}:
\begin{theorem}[Gromov-Eliashberg]\label{thm:gromov-eliashberg}
    For $\bm v\in C^1(\bar{\Omega},\RR^3)$, there exists $\{\bm v_n\}_{n=1}^\infty\subset C^1(\bar{\Omega},\RR^3)$ satisfying $\abs{\partial_1\bm v_n\times\partial_2\bm v_n}\neq 0$ for all $x\in\bar\Omega$ and $\bm v_n\rightarrow \bm v$ in $W^{1,p}(\Omega, \bbR^3)$. 
\end{theorem}
Since we will subsequently approximate immersions by piecewise affine maps, we need the following smooth version of the above result:
\begin{corollary}\label{cor:gromov-eliashberg-with-smooth}
    Let $1<p<\infty$ and $U\subset\RR^2$ be a bounded Lipschitz domain. For every $\bm v\in C^1(\bar{U},\RR^3)$, there exists $\{\bm v_n\}\subset C^\infty(\bar{U},\RR^3)$ with $\bm v_n\rightarrow\bm v$ in $W^{1,p}(U,\RR^3)$ and $\abs{\partial_1\bm v_n\times\partial_2\bm v_n}\neq0$ for all $x\in\bar{U}$.
\end{corollary}
\begin{proof}
    Fix $n$. By Theorem~\ref{thm:gromov-eliashberg}, there exists $\bm z\in C^1(\bar{U},\RR^3)$ with $\abs{\partial_1\bm z\times\partial_2\bm z}\neq0$ on $\bar{U}$ and $\norm{\bm z-\bm v}_{W^{1,p}(U)}<\frac{1}{2n}$. Since $\abs{\partial_1\bm z\times\partial_2\bm z}$ is continuous and strictly positive on $\bar{U}$, we have from the extreme value theorem
    \begin{align*}
        \alpha:=\min_{\bar{U}}\abs{\partial_1\bm z\times\partial_2\bm z}>0.
    \end{align*}
    As $U$ is Lipschitz, we may extend $\bm z$ to $\bar{\bm z}\in C^1_c(\RR^2,\RR^3)$. Denote the standard mollification kernel by $\rho_\varepsilon$. From uniform continuity, $\bar{\bm z}*\rho_{\varepsilon}\rightarrow\bar{\bm z}$ in $C^1(\RR^2,\RR^3)$ as $\varepsilon\rightarrow0$. Take $\varepsilon_n$ small enough so that $\bm v_n:=(\bar{\bm z}*\rho_{\varepsilon_n})|_{\bar{U}}\in C^\infty(\bar{U},\RR^3)$; hence
    \begin{align*}
        \abs{\abs{\partial_1\bm v_n\times\partial_2\bm v_n}-\abs{\partial_1\bm z\times\partial_2\bm z}}<\frac{\alpha}{2}\text{ on }\bar{U},\quad\norm{\bm v_n-\bm z}_{W^{1,p}(U)}<\frac{1}{2n}.
    \end{align*}
    Together this gives $\abs{\partial_1\bm v_n\times\partial_2\bm v_n}\geq\frac{\alpha}{2}>0$ and $\norm{\bm v_n-\bm v}_{W^{1,p}(U)}<\frac{1}{n}$.
\end{proof}
Next we show that every $p$-gradient Young measure can be generated by a $p$-equiintegrable sequence of piecewise affine immersions. First we define the relevant class of piecewise affine maps defined on regular triangulations of the plane:
\begin{definition}
    We denote by $\mathcal{T}_h$ the equilateral triangular mesh of $\RR^2$ with side length $h$. For a domain $\Omega\subset\RR^2$, we define the following set of piecewise affine maps
    \begin{align*}
        P_c(\Omega,\RR^3):=\{\bm v\in C(\Omega,\RR^3):\bm v|_{T\cap\Omega}\text{ is affine for each }T\in\mathcal{T}_h\text{ for some }h>0\}.
    \end{align*}
\end{definition}

\begin{proposition}\label{prop:Young-measures-generate-by-PC}
    Let \(1<p<\infty\) and let \(\Omega\subset\mathbb R^2\) be a bounded Lipschitz domain. For every  $\mu\in\mathcal{GY}^p(\Omega,\RR^{3\times2})$ there exists a sequence $\{\bm {v}^k\}\subset P_c(\Omega,\RR^3)$ such that $\{\nabla \bm {v}^k\}$ is $p$-equiintegrable, \(\nabla \bm {v}^k\) generates \(\mu\), and $ J(\nabla \bm {v}^k)>0$ a.e. in $\Omega$.
\end{proposition}
\begin{proof}
    This proof will be carried out using several levels of approximation.

    Let $\{\bm u^k\}\subset W^{1,p}(\Omega, \bbR^3)$ be a norm-bounded sequence that generates $\mu$ and such that $\{\nabla \bm u^k\}$ is $p$-equiintegrable. Assume $\bm u^k \weak \bm u$ for $\bm u \in W^{1,p}(\Omega, \bbR^3)$. From~\cite[Lemma 8.3]{pedregal_parametrized_1997}, we can assume that $\bm u^k-\bm u\in W^{1,p}_0$ and denote the extension by zero to all of $\RR^2$ by $\widetilde{\bm u^k-\bm u}$. Let $E:W^{1,p}(\Omega,\RR^3)\rightarrow W^{1,p}(\RR^2,\RR^3)$ be the Sobolev extension operator, then set 
    \begin{align*}
        \bar{\bm{u}}^k:=\widetilde{\bm u^k-\bm u}+E\bm u\in W^{1,p}(\RR^2,\RR^3).
    \end{align*}
    Then the sequence $\{\nabla\bar{\bm u}^k\}$ is $p$-equiintegrable on $\RR^2$ since it is the sum of an equiintegrable sequence and a fixed $W^{1,p}$ function. 
    
    Convolving with a standard mollification kernel $\rho_\e$, we get $\bar{\bm u}^k_\e=\bar{\bm u}^k\ast\rho_\e\in C^\infty(\RR^2,\RR^3)$. Note that
    \begin{align}\label{eq:mollification-L^p-norm}
        \abs{\nabla\bar{\bm u}^k_\e}^p=\abs{\nabla\bar{\bm u}^k\ast\rho_\e}^p\leq\abs{\nabla\bar{\bm u}^k}^p\ast\rho_\e.
    \end{align}
    For $\delta>0$, let $\Omega_\delta\supset\Omega$ so that $\dist(\partial\Omega_\delta,\Omega)\geq\delta$. Taking a diagonal sequence and restricting to $\bar\Omega_\delta$, posing $\bm w^k:=\bar{\bm u}^k_{\e_k}|_{\bar\Omega_\delta}$, we get the sequence $\{\bm w^k\}\subset C^\infty(\bar{\Omega}_\delta,\RR^3)$, with $\{\nabla\bm w^k\}$ which, by virtue of Lemma~\ref{lemma:strong-approximation-YM-and-equiintegrability}, is $p$-equiintegrable on $\Omega_\delta$ and generates the gradient Young measure $\hat{\mu}\in\mathcal{GY}^p(\Omega_\delta,\RR^{3\times2})$ whose restriction $\hat{\mu}|_{\Omega}=\mu$.

    Next, we use Corollary~\ref{cor:gromov-eliashberg-with-smooth} to approximate each $\bm w^k$ strongly in $W^{1,p}(\Omega_\delta,\RR^3)$ by a sequence $\tilde{\bm w}^k_i\in C^\infty(\bar{\Omega}_\delta,\RR^3)$ satisfying $J(\nabla\tilde{\bm w}^k_i)\neq0$, and from a further diagonalization argument, we obtain a sequence $\{\tilde{\bm w}^k\}$ satisfying $J(\nabla\tilde{\bm w}^k)\neq0$, whose gradients generate the same Young measure $\hat{\mu}$ and $\{\nabla\tilde{\bm w}^k\}$ is $p$-equiintegrable (again from Lemma~\ref{lemma:strong-approximation-YM-and-equiintegrability}). 

    Let $\mathcal{T}_h$ be the equilateral triangular mesh of size $h<\delta$ in $\RR^2$. For $\bm u\in C^\infty(\bar{\Omega}_\delta,\RR^3)$, let $\Pi_h \bm u$ denote the unique piecewise affine interpolant subordinate to the mesh $\mathcal{T}_h$. Now consider $\tilde{\bm w}^k\in C^\infty(\bar{\Omega}_\delta,\RR^3)$ with $J(\nabla\tilde{\bm w}^k)=\abs{\partial_1\tilde{\bm w}^k\times\partial_2\tilde{\bm w}^k}>0$. Since $\tilde{\bm w}^k$ is smooth on a compact set, there exists some $\e_k$ such that
    \begin{align*}
        \abs{\partial_1\tilde{\bm w}^k\times\partial_2\tilde{\bm w}^k}\geq\e_k>0.
    \end{align*}
    We denote 
    \begin{align*}
        \mathcal{T}_h(\Omega):=\{T\in\mathcal{T}_h:T\cap\Omega\neq\emptyset\}.
    \end{align*}
    Let $T\in\mathcal{T}_h(\Omega)$ be any triangle in the mesh, then we can write 
    \begin{equation}\label{eqfem}
        \abs{\nabla\Pi_h\bm\tilde{\bm w}^k-\nabla\bm\tilde{\bm w}^k}\leq C\norm{\nabla^2\bm\tilde{\bm w}^k}_{L^\infty}h
    \end{equation}
    for some $C>0$. To see this, let us fix \(k\). Let $T$ be an equilateral triangle of the mesh whose side length is \(h\), with vertices \(a_0,a_1,a_2\) and barycentric coordinate functions \(\lambda_0,\lambda_1,\lambda_2\). The nodal affine interpolant is given by
    \[
        (\Pi_h \bm\tilde{\bm w}^k )|_T
        = \sum_{i=0}^{2} \bm\tilde{\bm w}^{\,k}(a_i)\lambda_i.
    \]
    For any fixed \(x \in \operatorname{int}T\), consider the affine Taylor polynomial
    \[
        \mathbf p_x(y)
        := \bm\tilde{\bm w}^{\,k}(x) + \nabla\bm\tilde{\bm w}^{\,k}(x)(y-x).
    \]
    Taylor's theorem gives
    \[
        \bm\tilde{\bm w}^{\,k}(a_i)
        = \mathbf p_x(a_i) + \mathbf r_i,
        \qquad
        |\mathbf r_i|
        \le \frac12 |a_i-x|^2
            \|D^2\bm\tilde{\bm w}^{\,k}\|_{L^\infty(T)}
        \le \frac12 h^2
            \|D^2\bm\tilde{\bm w}^{\,k}\|_{L^\infty(T)}.
    \]
    Since nodal interpolation reproduces affine maps, \(\Pi_h\mathbf p_x = \mathbf p_x\). Hence, interpolating the Taylor expansion and differentiating with respect to \(y\), we obtain
    \[
        \nabla\Pi_h\bm\tilde{\bm w}^{\,k} - \nabla\bm\tilde{\bm w}^{\,k}(x)
        = \sum_{i=0}^{2} \mathbf r_i \otimes \nabla\lambda_i.
    \]
    The magnitude of \(\nabla\lambda_i\) is the reciprocal of the altitude from \(a_i\). Since \(T\) is equilateral,
    \[
        |\nabla\lambda_i| = \frac{2}{\sqrt{3}\,h},
        \qquad i=0,1,2.
    \]
    Consequently,
    \[
        \begin{aligned}
            \bigl|\nabla\Pi_h \tilde{\bm w}^{\,k} - \nabla\bm\tilde{\bm w}^{\,k}(x)\bigr|
            &\le \sum_{i=0}^{2}
                |\mathbf r_i|\,|\nabla\lambda_i| \\
            &\le 3 \cdot \frac12 h^2
                \|D^2 \tilde{\bm w}^{\,k}\|_{L^\infty(T)}
                \cdot \frac{2}{\sqrt{3}\,h} \\
            &\le \sqrt{3} h\,\|D^2 \tilde{\bm w}^{\,k}\|_{L^\infty(T)}.
        \end{aligned}
    \]
    
    We will rewrite \eqref{eqfem} as 
    \begin{align*}
        \partial_i\Pi_h\bm\tilde{\bm w}^k=\partial_i\bm\tilde{\bm w}^k+\bm{R}^k_i,\quad i=1,2,
    \end{align*}
    For fixed $k$, set $K_k:=\sqrt{3}\norm{D^2\tilde{\bm w}^k}_{L^\infty}$ and $M_k:=\norm{\nabla\tilde{\bm w}^k}_{L^\infty}$. The vector error term $\bm{R}^k_i$ then satisfies $\abs{\bm{R}^k_i}\leq K_kh$. Taking the cross product we get
    \begin{align*}
        \partial_1\Pi_h\bm\tilde{\bm w}^k\times\partial_2\Pi_h\bm\tilde{\bm w}^k=\partial_1\bm\tilde{\bm w}^k\times\partial_2\bm\tilde{\bm w}^k+\partial_1\bm\tilde{\bm w}^k\times\bm{R}^k_2+\bm{R}^k_1\times\partial_2\bm\tilde{\bm w}^k+\bm{R}^k_1\times\bm{R}^k_2.
    \end{align*}
    Taking norms and using the triangle inequality,
    \begin{align*}
        \abs{\partial_1\Pi_h\bm\tilde{\bm w}^k\times\partial_2\Pi_h\bm\tilde{\bm w}^k}&\geq\abs{\partial_1\bm\tilde{\bm w}^k\times\partial_2\bm\tilde{\bm w}^k}-\abs{\partial_1\bm\tilde{\bm w}^k\times\bm{R}^k_2}-\abs{\bm{R}^k_1\times\partial_2\bm\tilde{\bm w}^k}-\abs{\bm{R}^k_1\times\bm{R}^k_2}\\
        &\geq\abs{\partial_1\bm\tilde{\bm w}^k\times\partial_2\bm\tilde{\bm w}^k}-2M_kK_kh-K_k^2h^2.
    \end{align*}
    Since $\abs{\partial_1\bm\tilde{\bm w}^k\times\partial_2\bm\tilde{\bm w}^k}\geq\varepsilon_k>0$, we can find some $h_k$ small enough so that $2M_kK_kh_k+K_k^2h_k^2<\varepsilon_k$. We then have $J(\nabla\Pi_{h_k}\tilde{\bm w}^k)>0$ and taking $h_k$ smaller if necessary, we also have $\norm{\nabla\Pi_{h_k}\tilde{\bm w}^k-\nabla\tilde{\bm w}^k}_{L^p}\rightarrow0$. Setting $\bm v^k:=\Pi_{h_k}\tilde{\bm w}^k$, we obtain a sequence of piecewise affine immersions. Finally, applying Lemma~\ref{lemma:strong-approximation-YM-and-equiintegrability} once more, $\{\nabla\bm v^k\}$ satisfy the $p$-equiintegrability condition and generate the Young measure $\hat{\mu}_{\Omega}=\mu$.
\end{proof}

We now prove our main result on Young measures, Theorem~\ref{thm:gym}. As outlined earlier, the idea is to modify the generating maps on the ``bad triangles'' where the Jacobian is small. Since Lemma~\ref{lemma:replacement-inside-triangle} applies only to strictly short maps, on each bad triangle we first apply a uniform compression of the target, perform the replacement, and then undo the compression.

\begin{proof}[Proof of Theorem \ref{thm:gym}]
    Since $\mathcal{JY}^{p,-q}(\Omega,\RR^{3\times2})\subset\mathcal{GY}^{p}(\Omega,\RR^{3\times2})$, Proposition~\ref{prop:Young-measures-generate-by-PC}, provides a sequence of piecewise affine immersions $\{\bm v^k\}\subset P_c(\Omega,\RR^3)$ such that $\{\nabla \bm v^k\}$ is $p$-equiintegrable and generates $\nu$. We write $\mathcal{T}^k$ for the mesh on which $\bm v^k$ is piecewise affine and recall from the construction that $\bm v^k$ is the restriction to $\Omega$ of a map which is affine on each \textit{whole} triangle $T\in\mathcal{T}^k(\Omega)$.

    Define the truncated function $\Phi_\e$:
    \begin{align*}
        \Phi_\e(\bm A):=\abs{\bm A}^p+\left(\abs{J(\bm A)}\vee\e\right)^{-q}.
    \end{align*}
    Since $\{\nabla \bm v^k\}$ is $p$-equiintegrable, and $\bm A\mapsto\left(\abs{J(\bm A)}\vee\e\right)^{-q}$ is bounded and continuous, for fixed $\e$ we have that
    \begin{align*}
        \int_\Omega\Phi_\e(\nabla \bm v^k)\dif x\rightarrow\int_\Omega\int_{\RR^{3\times2}}\Phi_\e(\bm F)\dif\nu_x(\bm F)\dif x\leq\int_\Omega\int_{\RR^{3\times2}}\Phi(\bm F)\dif\nu_x(\bm F)\dif x<+\infty.
    \end{align*}
    We can partition $\Omega$ into the sets $G^k_\e:=\{J(\nabla \bm v^k)>\e\}$ and $B^k_\e:=\Omega\setminus G^k_\e=\{J(\nabla \bm v^k)\leq\e\}$. Then,
    \begin{align*}
        \int_{\Omega}\Phi_\e(\nabla \bm v^k)\dif x = \int_\Omega\abs{\nabla \bm v^k}^p\dif x + \int_{G^k_\e}\frac{1}{\abs{J(\nabla \bm v^k)}^q}\dif x + \int_{B^k_\e}\frac{1}{\e^q}\dif x\leq C<+\infty,
    \end{align*}
    thus, $\abs{B^k_\e}\leq C\e^q$. The idea now is to modify
    $\bm v^k$ appropriately on $B^k_\e$.

    Since $\bm v^k$ is affine on each $T\in\mathcal{T}^k(\Omega)$, the Jacobian $J(\nabla\bm v^k)$ is constant on each such triangle, so that
    \begin{align*}
        B_{\varepsilon}^k=\Omega\cap\bigcup_{T\in\mathcal{B}^k_\varepsilon}T,\qquad \mathcal{B}_{\varepsilon}^k:=\{T\in\mathcal{T}^k(\Omega):J(\nabla\bm v^k|_{T})\leq\varepsilon\},
    \end{align*}
    and $\mathcal{B}^k_\varepsilon$ is a finite collection. If $T$ meets $\partial\Omega$, the intersection $\Omega\cap T$ need not be a triangle, so we make the replacement on the entirety of $T$, on which $\bm v^k$ is affine, and then restrict to $T\cap\Omega$ in the end. We focus on one of these triangles, $T_i\in\mathcal{B}^k_\e$, where $\bm v^k_i:=\bm v^k|_{T_i}$ is affine. Let $\sigma^{(1)}_i,\sigma^{(2)}_i$ denote the singular values of $\nabla\bm v^k_i$, labelled so that $\sigma_i^{(1)}\geq\sigma_i^{(2)}$ and recall that $\sigma_i^{(1)}=\textrm{Lip}(\bm v^k_i)$.

    Since $T_i\in \mathcal{B}^k_\e$, $0<\sigma^{(1)}_i\sigma^{(2)}_i\leq\e$. Apply the uniform compression $\Lambda:\bm y\mapsto \frac{1}{1+\sigma^{(1)}_i}\bm y$, so that the composed map, $\tilde{\bm v}^k_i:=\Lambda\circ\bm v^k_i$ is a non-singular, strictly short affine map. Let $\delta_i>0$ and use Lemma~\ref{lemma:replacement-inside-triangle} with $a_0=\pi/3$ to replace $\tilde{\bm v}^k_i$ with a map $\tilde{\bm z}^k_i$ satisfying $\textrm{Lip}(\tilde{\bm z}^k_i)\leq C_1$, $J(\nabla\tilde{\bm z}^k_i)\geq J(\nabla\tilde{\bm v}^k_i)$ on $T_i$ and $J(\nabla\tilde{\bm z}^k_i)=1$ on the shrunk triangle $T^{\delta_i}_i:=D^{1-\delta_i}T_i$. Undoing the compression, the map $\bm z^k_i:=\Lambda^{-1}\circ\tilde{\bm{z}}^k_i$ satisfies
    \begin{align}\label{eq:replacement-boundary-condition}
        \bm z^k_i=\bm v^k_i \text{ on }\partial T_i,
    \end{align}
    \begin{align}\label{eq:replacement-Lipschitz-constant}
        \textrm{Lip}(\bm z^k_i)\leq C_1(1+\sigma^{(1)}_i)=C_1(1+\textrm{Lip}(\bm v^k_i)),
    \end{align}
    \begin{align}\label{eq:replacement-J-on-T}
        J(\nabla \bm z^k_i)\geq J(\nabla\bm v^k_i) \text{ on }T_i
    \end{align}
    \begin{align}\label{eq:replacement-J-on-Tdelta}
        J(\nabla\bm z^k_i)=(1+\sigma^{(1)}_i)^2\geq 1 \text{ on }T^{\delta_i}_i. 
    \end{align}
    
    Setting $J(\nabla\bm v^k_i)= m_i>0$ we have
    \begin{align*}
        \int_{T_i\cap\Omega}\frac{1}{\abs{J(\nabla \bm z^k_i)}^q}\dif x&\leq\int_{T_i\setminus T^{\delta_i}_i}\frac{1}{\abs{J(\nabla \bm v^k_i)}^q}\dif x+\int_{T^{\delta_i}_i\cap\Omega}\frac{1}{(1+\sigma^{(1)}_i)^{2q}}\dif x\\
        &\leq\frac{1}{m_i^q}\abs{T_i\setminus T^{\delta_i}_i}+\abs{T_i\cap\Omega}.
    \end{align*}
    Since $\abs{T_i\setminus T^{\delta_i}_i}=(2\delta_i-\delta_i^2)\abs{T_i}$ and $\abs{T_i\cap\Omega}>0$, we can choose $\delta_i>0$ small enough so that
    \begin{align*}
        \frac{(2\delta_i-\delta_i^2)\abs{T_i}}{m_i^q}\leq \abs{T_i\cap\Omega},
    \end{align*}
    so that we have
    \begin{align}\label{eq:J-Lq-bound-on-triangle}
        \int_{T_i\cap\Omega}\frac{1}{\abs{J(\nabla \bm z^k_i)}^q}\dif x\leq 2\abs{T_i\cap\Omega}.
    \end{align}
     Making similar replacements in each $T_i\in \mathcal{B}^k_\e$, we obtain the map 
    \begin{align*}
        \bm w^k_\e(x):=\begin{cases}
            \bm v^k(x)&x\in G^k_\e\\
            \bm z^k_i(x)&x\in T_i\cap\Omega,T_i\in \mathcal{B}^k_\e,
        \end{cases},
    \end{align*}
    which is continuous from~\eqref{eq:replacement-boundary-condition}. Furthermore, from~\eqref{eq:replacement-boundary-condition} and~\eqref{eq:replacement-Lipschitz-constant}, it follows that $\bm w^k_\e\in W^{1,p}$, and
    \begin{align}\label{eq:Phi-upper-bound-1}
        \int_\Omega\Phi(\nabla\bm w^k_\e)\dif x &= \int_\Omega\abs{\nabla \bm w^k_\e}^p\dif x+\int_\Omega\frac{1}{\abs{J(\nabla\bm w^k_\e)}^q}\dif x\nonumber\\
        &=\int_{G^k_\e}\abs{\nabla\bm v^k}^p\dif x + \int_{G^k_\e}\frac{1}{\abs{J(\nabla\bm v^k)}^q}\dif x\nonumber\\
        &\quad +\sum_{i}\int_{T_i\cap\Omega}\abs{\nabla\bm w^k_\e}^p\dif x+\sum_{i}\int_{T_i\cap\Omega}\frac{1}{\abs{J(\nabla\bm w^k_\e)}^q}\dif x
    \end{align}
    From~\eqref{eq:replacement-Lipschitz-constant},
    \begin{align}\label{eq:lipschitz-bound-on-triangles}
        \sum_{i}\int_{T_i\cap\Omega}\abs{\nabla\bm w^k_\e}^p\dif x&\leq\sum_i\int_{T_i\cap\Omega}C_1^p\abs{1+\textrm{Lip}(\bm v_i^k)}^p\dif x\nonumber\\
        &\leq C_3\sum_i\left(\abs{T_i\cap\Omega}+\int_{T_i\cap\Omega}\abs{\nabla\bm v^k_i}^p\dif x\right)\nonumber\\
        &=C_3\left(\abs{B^k_\e}+\int_{B^k_\e}\abs{\nabla\bm v^k}^p\dif x\right)
    \end{align}
    for some constant $C_3>0$. Combining the above with~\eqref{eq:J-Lq-bound-on-triangle} and~\eqref{eq:Phi-upper-bound-1},
    \begin{align*}
        \int_\Omega\Phi(\nabla\bm w^k_\e)\dif x\leq C_4\left(\int_\Omega\abs{\nabla\bm v^k}^p\dif x+\int_{G^k_\e}\frac{1}{\abs{J(\nabla\bm v^k)}^q}\dif x+\abs{B^k_\e}\right)\leq C_5\left(1+\abs{B^k_\e}\right)
    \end{align*}
    for constants $C_4$, $C_5>0$. We pick a diagonal sequence $\bm u^k:=\bm w^k_{\e_k}$ with $\e_k\rightarrow 0$, and notice that $\{x\in\Omega:\nabla\bm v^k\neq\nabla\bm u^k\}\subset B^k_{\e_k}$. Since $\abs{B^k_{\e_k}}\leq C\e_k^q\rightarrow 0$ as $k\rightarrow\infty$, we conclude by Lemma~\ref{lemma:strong-approximation-YM-and-equiintegrability} that $\{\nabla\bm u^k\}$ generates the same Young measure as $\{\nabla\bm v^k\}$.

    We turn to equiintegrability. Notice that since $\bm w^k_\e=\bm v^k$ on $G^k_\e$ and $\bm w^k_\e=\bm z^k_i$ on each $T_i\cap\Omega$ with $T_i\in \mathcal{B}^k_\e$,
    \begin{align}\label{eq:YM-equiintegrability-I}
        \int_{\Omega}\Phi(\nabla\bm{w}^k_\e)\dif x=\int_{G^k_\e}\Phi(\nabla\bm v^k)\dif x+\sum_i\int_{T_i\cap\Omega}\abs{\nabla\bm z^k_i}^p\dif x+\sum_i\int_{T_i\cap\Omega}\abs{J(\nabla\bm z^k_i)}^{-q}\dif x
    \end{align}
    On $G^k_\e$, since $\Phi=\Phi_\e$ we can write
    \begin{align*}
        \int_{G^k_\e}\Phi(\nabla \bm v^k)\dif x = \int_{G^k_\e}\Phi_\e(\nabla\bm v^k)\dif x\leq\int_\Omega\Phi_\e(\nabla\bm v^k)\dif x.
    \end{align*}
    Since $\Phi_\e(\bm F)\geq\abs{\bm F}^p\geq0$,
    \begin{align}\label{eq:YM-equiintegrability-II}
        \int_{G^k_\e}\Phi(\nabla \bm v^k)\dif x \leq\int_{\Omega}\Phi_\e(\nabla\bm v^k)\dif x.
    \end{align}
    From~\eqref{eq:lipschitz-bound-on-triangles}, we have for some $C_3>0$,
    \begin{align}\label{eq:YM-equiintegrability-III}
        \sum_i\int_{T_i\cap\Omega}\abs{\nabla\bm z^k_i}^p\dif x\leq C_3\left(\abs{B^k_\e}+\int_{B^k_\e}\abs{\nabla\bm v^k}^p\dif x\right).
    \end{align}
    Then using~\eqref{eq:J-Lq-bound-on-triangle},~\eqref{eq:YM-equiintegrability-II}, and~\eqref{eq:YM-equiintegrability-III} in~\eqref{eq:YM-equiintegrability-I}, we have
    \begin{align}\label{eq:YM-equiintegrability-IV}
        \int_\Omega\Phi(\nabla\bm w^k_\e)\dif x\leq\int_\Omega\Phi_\e(\nabla\bm v^k)\dif x+C_3\int_{B^k_\e}\abs{\nabla\bm v^k}^p\dif x+C_3\abs{B^k_\e}.
    \end{align}
    We will take a limit of the above along an appropriately chosen subsequence $\e_k\rightarrow0$. First, from the $p$-equiintegrability of $\{\nabla \bm v^k\}$, as $k\rightarrow\infty$
    \begin{align*}
        \int_\Omega\Phi_\e(\nabla\bm v^k)\dif x\rightarrow\int_\Omega\int_{\RR^{3\times 2}}\Phi_\e(\bm F)\dif\nu_x(\bm F)\dif x.
    \end{align*}
    Second, from the monotone convergence theorem, as $\e\rightarrow0$
    \begin{align*}
        \int_\Omega\int_{\RR^{3\times2}}\Phi_\e(\bm F)\dif\nu_x(\bm F)\dif x\rightarrow\int_\Omega\int_{\RR^{3\times2}}\Phi(\bm F)\dif\nu_x(\bm F)\dif x
    \end{align*}
    Then for an appropriate sequence $\e_k\rightarrow0$, and using $\abs{B^k_{\e_k}}\rightarrow0$ (in conjunction with the $p$-equiintegrability of $\{\nabla\bm v^k\}$), we can take the limit in~\eqref{eq:YM-equiintegrability-IV} to get
    \begin{align*}
         \limsup_{k\rightarrow\infty}\int_\Omega\Phi(\nabla\bm u^k)\dif x\leq\int_\Omega\int_{\RR^{3\times2}}\Phi(\bm F)\dif\nu_x(\bm F)\dif x.
    \end{align*}
    Of course, we also have from the fundamental property of Young measures (Theorem~\ref{thm:fundamental-property-of-Young-measures}) 
    \begin{align*}
        \liminf_{k\rightarrow\infty}\int_\Omega\Phi(\nabla\bm u^k)\dif x\geq\int_\Omega\int_{\RR^{3\times2}}\Phi(\bm F)\dif\nu_x(\bm F)\dif x.
    \end{align*}
    Since $\{\nabla\bm u^k\}$ generates $\nu$, we can conclude the equiintegrability of $\{\Phi(\nabla\bm u^k)\}$, i.e.,
    \begin{align*}
        \lim_{k\rightarrow\infty}\int_\Omega\Phi(\nabla\bm u^k)\dif x=\int_\Omega\int_{\RR^{3\times2}}\Phi(\bm F)\dif\nu_x(\bm F)\dif x.
    \end{align*}
\end{proof}

\subsection{Smooth generation}

In the next section we apply Theorem~\ref{thm:gym} to construct recovery sequences for the dimension reduction problem, where it is important that the generating sequences be smooth. To regularize the piecewise affine sequence produced by Theorem~\ref{thm:gym}, we will use the following result of Conti and Dolzmann~\cite[Proposition 4.1]{conti_derivation_2006}:
\begin{proposition}[Conti-Dolzmann]\label{prop:conti-dolzmann}
    Let $\Omega\subset\RR^2$ be a Lipschitz domain, $\Gamma\subset\bar{\Omega}$ a connected set which contains $\partial\Omega$ and define for $\eta>0$, the set $\Gamma^\eta:=\{x\in\bar{\Omega}:\dist(x,\Gamma)<\eta\}$. If $\bm u\in C^\infty(\bar{\Omega},\RR^3)$ and $\mathrm{rank}(\nabla \bm u)=2$ on $\Omega\setminus\Gamma^\eta$, then for any $\delta>0$, there exists $\bm w\in C^\infty(\bar{\Omega},\RR^3)$ such that $\norm{\bm u-\bm w}_{C^0(\Omega)}\leq\delta$,
    \begin{align*}
        \abs{\nabla \bm w}\leq c(\abs{\nabla \bm u}+1),\quad\abs{\partial_1\bm w\times\partial_2\bm w}\geq c\abs{\partial_1\bm u\times\partial_2\bm u}
    \end{align*}
    and
    \begin{align*}
        \bm w=\bm u\text{ on }\bar\Omega\setminus\Gamma^{2\eta},\quad\abs{\partial_1\bm w\times\partial_2\bm w}\geq c\text{ on }\Gamma^\eta.
    \end{align*}
All constants are absolute constants.
\end{proposition}

We will now prove Proposition \ref{prop:smooth-YM-generation}. The proof relies on mollifying the piecewise affine sequences produced by Theorem~\ref{thm:gym}. Since mollification may destroy the nondegeneracy of the Jacobian near the edges of the triangulation, we repair the mollified maps there using Proposition~\ref{prop:conti-dolzmann}.

\begin{proof}[Proof of Proposition \ref{prop:smooth-YM-generation}]
    Let $\{\bm v^k\}$ be a sequence generating $\nu$ given by Theorem~\ref{thm:gym}. Each $\bm v^k$ is piecewise affine on a finite triangulation. Denote the union of the 1-skeleton of the triangulation and $\partial\Omega$ by $\Gamma_k$. After extending the maps to $\RR^2$, mollify $\bm v^k$ with a standard mollification kernel $\rho_{\eta_k}$ with support in $B_{\eta_k}(0)$, and let $\tilde{\bm v}^k:=\bm v^k*\rho_{\eta_k}$. Note that since the maps are piecewise affine, for $x\notin\Gamma_k^{\eta_k}:=\{x:\dist(x,\Gamma_k)<\eta_k\}$, $\tilde{\bm v}^k(x)=\bm v^k(x)$.

    For each fixed $k$, the set $\Gamma_k$ is a finite graph and therefore has zero Lebesgue measure, hence $\abs{\Gamma^{3\eta}_k}\rightarrow0$ as $\eta\rightarrow0$. Furthermore, since $\Phi(\nabla\bm v^k)\in L^1(\Omega)$, it follows from the absolute continuity of the Lebesgue integral that $\int_{\Gamma^{3\eta}_k}\Phi(\nabla\bm v^k)\dif x\rightarrow0$ as $\eta\rightarrow0$. We may therefore choose $\eta_k>0$ sufficiently small so that 
    $$
        \abs{\Gamma_k^{3\eta_k}}\leq\frac{1}{k}, \qquad \int_{\Gamma_k^{3\eta_k}}\Phi(\nabla\bm v^k)\,\dif x\le \frac1k.
    $$ 
    As a result, $\abs{\{\nabla\bm v^k\neq\nabla \tilde{\bm v}^k\}}\rightarrow 0$ as $k\rightarrow\infty$, and therefore from Lemma~\ref{lemma:strong-approximation-YM-and-equiintegrability}, $\nabla\bm v^k$ and $\nabla\tilde{\bm v}^k$ generate the same Young measure. 
    Also, since $J(\nabla\bm v^k)>0$ a.e., $\textrm{rank}(\nabla\tilde{\bm v}^k(x))=2$ for $x\in\Omega\setminus\Gamma_k^{\eta_k}$.

    Applying Proposition~\ref{prop:conti-dolzmann} to $\tilde{\bm v}^k$ with $\eta=\eta_k$, we get $\bm u^k$ such that $\bm u^k=\tilde{\bm v}^k$ on $\bar{\Omega}\setminus\Gamma^{2\eta_k}_k$. Thus, $\abs{\{\nabla \bm u^k\neq\nabla\tilde{\bm v}^k\}}\leq\abs{\Gamma^{2\eta_k}_k}\rightarrow 0$ as $k\rightarrow\infty$ and $\{\nabla \bm u^k\}$ generates the same Young measure $\nu$. Also,
    $\abs{\nabla \bm u^k}\leq c(\abs{\nabla \tilde{\bm v}^k}+1)$, $J(\nabla\bm u^k)\geq c J(\nabla \tilde{\bm v}^k)$ and importantly $J(\nabla\bm u^k)\geq c$ on $\Gamma^{\eta_k}_k$. Now we compute $\int_\Omega\Phi(\nabla\bm u^k)\dif x$: 

    On $\Gamma^{2\eta_k}_k$, we have
    \begin{align*}
        \int_{\Gamma^{2\eta_k}_k}\abs{\nabla\bm u^k}^p\dif x &\leq  c\int_{\Gamma^{2\eta_k}_k}\left(\abs{\nabla\tilde{\bm v}^k}^p+1\right)\dif x\\
        &\leq c\int_{\Gamma_k^{2\eta_k}}\int_{\mathbb R^2}|\nabla\bm v^k(y)|^p\rho_{\eta_k}(x-y)\,\dif y\dif x + c|\Gamma^{2\eta_k}_k|\\
        &=c\int_{\mathbb R^2}|\nabla\bm v^k(y)|^p\left(\int_{\Gamma_k^{2\eta_k}}\rho_{\eta_k}(x-y)\,\dif x\right)\dif y+ c|\Gamma^{2\eta_k}_k|\\
        &\leq c \left(\int_{\Gamma^{3\eta_k}_k}\abs{\nabla{\bm v}^k(y)}^p\dif y + |\Gamma^{2\eta_k}_k|\right)
    \end{align*}
    since the inner integral in the second-to-last line vanishes whenever \(y\notin\Gamma_k^{3\eta_k}\). The sequence $\{\nabla \bm v^k\}$ is $p$-equiintegrable, and, by our choice of $\eta_k$, the right-hand side tends to zero. 
 
    On the set $\Gamma^{\eta_k}_k$, we have
    \begin{align*}
        \int_{\Gamma^{\eta_k}_k}(J(\nabla\bm u^k))^{-q}\dif x\leq c^{-q}\abs{\Gamma^{\eta_k}_k},
    \end{align*}
    which also approaches zero. On the set $\Gamma^{2\eta_k}_k\setminus\Gamma^{\eta_k}_k$, $J(\nabla\bm u^k)\geq cJ(\nabla\tilde{\bm{v}}^k)=cJ(\nabla\bm v^k)$. Then 
    \begin{align*}
        \int_{\Gamma^{2\eta_k}_k\setminus\Gamma^{\eta_k}_k}(J(\nabla\bm u^k))^{-q}\dif x\leq c^{-q}\int_{\Gamma^{2\eta_k}_k\setminus\Gamma^{\eta_k}_k}J(\nabla\bm v^k)^{-q}\dif x,
    \end{align*}
    which, from the equiintegrability of $\Phi(\nabla\bm v^k)$, also approaches zero as $k\rightarrow\infty$. On the remaining set, i.e. $\Omega\setminus\Gamma^{2\eta_k}_k$, we have $\nabla\bm u^k=\nabla\bm v^k$, therefore~\eqref{eq:Cinfty-Phi-equiintergability} follows, i.e.,
    \begin{align*}
        \lim_{k\rightarrow\infty}\int_\Omega\Phi(\nabla\bm u^k)\dif x=\int_\Omega\int_{\RR^{3\times 2}}\Phi(\bm F)\dif\nu_x(\bm F)\dif x.
    \end{align*}
    
\end{proof}

\section{Proof of Theorem \ref{thm:gc}}\label{sec:problem-formulation-Young}
\subsection{Growth bounds}

To characterize the limit problem as $\e\rightarrow0$, we define $W_0:\Omega\times\RR^{3\times 2}\rightarrow[0,+\infty]$ as
\begin{align}
    \label{eqn:dimension-reduced-energy-density}
    W_0(x',\bar{\bm F}):=\inf_{\bm z\in\RR^3}\{W(x',(\bar{\bm F}|\bm z)), \ \det (\bar{\bm F}|\bm z) > 0\}.
\end{align}
Note that $W_0$ takes the value $+\infty$ precisely on matrices $\bar{\bm F}\in\RR^{3\times2}$ of rank less than two (since no admissible $\bm z$ exists for such matrices).
    
We first prove the following lemma that allows us to show that $W_0$ satisfies certain desirable growth conditions:
\begin{lemma}
    \label{lem:solution-to-optimization-problem}
    Let $\bm x, \bm y \in \RR^n\setminus\{0\}$ be such that $\bm x\cdot \bm y > 0$, and let $p,s \in (0, \infty)$. For given $\bm y$, the solution to
			\begin{equation}\label{eqn:x-y-relation}
				\abs{\bm x}^{p-2} (\bm y\cdot \bm x)^{s+1} \bm x = \frac{s}{p} \bm y
			\end{equation}
			is
			\begin{align*}
                \bm x = \left(\frac{s}{p} \abs{\bm y}^{-p-2s }\right)^{1/(p+s)}\bm y
            \end{align*}
\end{lemma}
\begin{proof}
    Suppose $\bm x$ solves~\eqref{eqn:x-y-relation} then by projecting~\eqref{eqn:x-y-relation} along $\bm x$ and $\bm y$, we see that
    \begin{align*}
        p\abs{\bm x}^p=s(\bm y\cdot \bm x)^{-s},\qquad p\abs{\bm x}^{p-2}=s(\bm y\cdot \bm x)^{-s-2}\abs{\bm y}^2.
    \end{align*}
    Taking the ratio of the above two equations (since $\bm x,\bm y\neq \bm 0$ and $\bm x\cdot \bm y>0$), it is necessary that,
    \begin{align*}
        \frac{\bm x}{\abs{\bm x}}\cdot\frac{\bm y}{\abs{\bm y}}=1.
    \end{align*}
    Thus $\bm x=\alpha \bm y$ for some $\alpha>0$, and plugging this into~\eqref{eqn:x-y-relation}, we get
    \begin{align*}
        \bm 0=\left(\alpha^{p-2}\abs{\bm y}^{p-2}(\alpha \bm y \cdot \bm y)^{s+1}\alpha - \frac{s}{p} \right)\bm y = \left(\alpha^{p+s}\abs{\bm y}^{p-2}(\bm y \cdot \bm y)^{s+1} - \frac{s}{p} \right)\bm y = \left(\alpha^{p+s}\abs{\bm y}^{p+2s} - \frac{s}{p} \right)\bm y,
    \end{align*}
    thus
    \begin{align*}
        \alpha = \left(\frac{s}{p}\abs{\bm y}^{-p-2s}\right)^{1/(p+s)},
    \end{align*}
    which yields the desired solution.
\end{proof}

With Lemma~\ref{lem:solution-to-optimization-problem} in hand, we have the following proposition that establishes two-sided bounds on $W_0$:
\begin{proposition}
    \label{prop:W0-growth-condition}
    Suppose $W$ satisfies~\eqref{itm:local-injectivity} and~\eqref{itm:dimension-reduction-growth}. Then $W_0$ is a Carath\'eodory integrand on $\Omega\times\RR^{3\times2}$ and satisfies
    \begin{align}\label{eqn:W0-growth-condition}
        C_1\left(c_p\left(\abs{\bar{\bm F}}^p+\frac{1}{J(\bar{\bm F})^{\frac{ps}{p+s}}}\right)   - C_2\right) \leq W_0(\cdot, \bar{\bm F})\leq 2 C_3\left(C_p\left(\abs{\bar{\bm{F}}}^p+\frac{1}{J(\bar{\bm F})^{\frac{ps}{p+s}}}\right)+C_4\right)
    \end{align}
    where $J(\bar{\bm F}):=\left(\det\bar{\bm F}^T\bar{\bm F}\right)^{1/2}$ and
    \[c_p := 2^{\min\{p/2-1,0\}}, \qquad C_p := 2^{\max\{p/2-1,0\}}.\]
\end{proposition}
\begin{proof}
    Consider $\bm F=(\bar{\bm F}|\bm z)\in\RR^{3\times3}$, where $\bar{\bm F}\in\RR^{3\times2}$ and $\bm z\in\RR^3$. Let the $2\times2$ subdeterminants of $\bar{\bm F}$ be denoted
    \begin{align}\label{eqn:subdeterminants3x2}
        m_1 = \det\begin{pmatrix}
            F_{21}&F_{22}\\
            F_{31}&F_{32}
        \end{pmatrix}\quad m_2 = -\det\begin{pmatrix}
            F_{11}&F_{12}\\
            F_{31}&F_{32}
        \end{pmatrix}\quad m_3 = \det\begin{pmatrix}
            F_{11}&F_{12}\\
            F_{21}&F_{22}
        \end{pmatrix}.
    \end{align}
    Let $\bm m=(m_1,m_2,m_3)^T$, then $\det\bm F = \bm m\cdot \bm z$. Moreover, 
    \begin{align*}
        J(\bar{\bm F}):=\left(\det\bar{\bm F}^T\bar{\bm F}\right)^{1/2}=\abs{\bm m}=\abs{\bm F\bm e_1\times\bm F\bm e_2}.
    \end{align*}
    We recall that for all $a,b \geq 0$ we have
    \[c_p (a^p+b^p)\leq (a^2+b^2)^{p/2} \leq C_p (a^p+b^p).\]
    
    Thus from~\eqref{itm:local-injectivity},
    \begin{align}\label{eq:W_0-lower-bound}
        W_0(\cdot, \bar{\bm F}) &=: \inf_{\bm z \in \RR^3} \left\{W(\cdot,(\bar{\bm F}|\bm z)), \ \det (\bar{\bm F}|\bm z) > 0\right\}\nonumber\\
				& \geq C_1\inf_{\bm z \in \RR^3}  \left\{\frac{1}{(\bm m\cdot \bm z)^s} + c_p(\abs{\bar{\bm F}}^p + \abs{\bm z}^p) - C_2 , \ \det (\bar{\bm F}|\bm z) > 0\right\}\nonumber\\
				& = C_1\inf_{\bm z \in \RR^3}  \left\{\frac{1}{(\bm m\cdot \bm z)^s}  + c_p\abs{\bm z}^p , \ \det (\bar{\bm F}|\bm z) > 0\right\} + C_1(c_p|\bar{\bm F}|^p - C_2).
    \end{align}
    A direct computation shows that if $\bm z$ optimizes the above, then
    \begin{align*}
        p c_p \abs{\bm z}^{p-2} \bm z = s (\bm m\cdot \bm z)^{-s-1}\bm m
    \end{align*}
    and by Lemma~\ref{lem:solution-to-optimization-problem},
    \begin{align*}
        \bm z = \bm m \left(\frac{s}{p c_p}\right)^{\frac{1}{p+s}} \abs{\bm m}^{\frac{-p-2s}{p+s}}.
    \end{align*}
    Accordingly
    \[|\bm z|^p = \left(\frac{s}{p c_p}\right)^{\frac{p}{p+s}} \abs{\bm m}^{\frac{-ps}{p+s}}, \qquad (\bm m \cdot \bm z)^{-s} = \left(\frac{s}{p c_p}\right)^{\frac{-s}{p+s}} \abs{\bm m}^{\frac{-ps}{p+s}}.\]
    Using the above in~\eqref{eq:W_0-lower-bound},
    \begin{align*}
        W_0(\cdot,\bar{\bm F})&\geq C_1\left[\left(c_p\left(\frac{s}{p c_p}\right)^{\frac{p}{p+s}}+\left(\frac{s}{p c_p}\right)^{-\frac{s}{p+s}}\right)\frac{1}{\abs{\bm m}^{\frac{ps}{p+s}}}+c_p\abs{\bar{\bm F}}^p-C_2\right]\\
        &=C_1\left[\frac{p+s}{p}\left(\frac{p c_p}{s}\right)^{s/(p+s)}\abs{\bm m}^{-\frac{ps}{p+s}}+c_p\abs{\bar{\bm F}}^p-C_2\right]\\
        &\geq C_1 \left(c_p\left(\frac{1}{\abs{\bm m}^{\frac{ps}{p+s}}}+\abs{\bar{\bm F}}^p\right)-C_2\right)
    \end{align*}
    since
    \[\frac{p+s}{p}\left(\frac{p c_p}{s}\right)^{s/(p+s)} \geq c_p \qquad \forall p,s > 0.\]
    For the upper bound, reasoning as before, we obtain
\begin{align*}
    W_0(\cdot,\bar{\bm F}) &\leq C_3\left[\frac{p+s}{p}\left(\frac{p C_p}{s}\right)^{s/(p+s)}\abs{\bm m}^{-\frac{ps}{p+s}}+C_p\abs{\bar{\bm F}}^p+C_4\right]\\
    & \leq C_3\left[2C_p \abs{\bm m}^{-\frac{ps}{p+s}}+C_p\abs{\bar{\bm F}}^p+C_4\right]\\
    &\leq 2 C_3\left[C_p \left(\abs{\bm m}^{-\frac{ps}{p+s}}+\abs{\bar{\bm F}}^p\right)+C_4\right]
\end{align*}
since
\[\frac{p+s}{p}\left(\frac{p C_p}{s}\right)^{s/(p+s)}\leq 2C_p \qquad \forall p,s > 0.\]
\end{proof}

\sloppy
\begin{remark}\label{rem:examples-of-polyconvex-neohookean}
    We point out that for some examples of $W$ satisfying~\eqref{itm:local-injectivity},~\eqref{itm:dimension-reduction-growth} and polyconvexity, the corresponding dimensionally reduced function $W_0$ is convex in the pair $(\bar{\bm F},J)$, which is the membrane convexity condition identified in~\cite{healey2025nonlinearly}. Consider for example, the compressible neo-Hookean bulk energy 
    \begin{align*}
        W(\bm F)=\begin{cases}
            a\abs{\bm F}^2+b(\det\bm F)^{-s} &\det\bm F>0\\
            +\infty &\det\bm F\leq 0
        \end{cases}
    \end{align*}
    with $a,b>0$ and $s>1$. A simple calculation shows that $W_0(\bar{\bm F})=a_0\abs{\bar{\bm F}}^2+b_0\abs{J}^{-\frac{2s}{s+2}}$ with $a_0,b_0>0$, which is clearly convex in the pair $(\bar{\bm F},J)$.
\end{remark}
\fussy

\subsection{Compactness}

We now prove that sequences of bulk deformations with bounded energy as $\e\rightarrow0$ are precompact in the sense of nondegenerate gradient Young measures. Note that item (i) below corresponds to item (i) of Theorem \ref{thm:gc}. Recall that $\mathsf{R}$ denotes the average-projection operator defined in Section~\ref{sec:main-results}.

\begin{proposition}\label{thm:Young-measure-compactness}
    Let $\e_k\rightarrow 0$ and recall that $\Omega_1=\Omega\times\left(-\frac{1}{2},\frac{1}{2}\right)\subset\RR^3$. Suppose $\{\mu^k\}\subset L^\infty_w(\Omega_1,\mathcal{M}(\RR^{3\times3}))$ is a sequence such that $\sup_k \mathcal{E}_{\e_k}[\mu^k]<+\infty$. Since $\mathcal{E}_{\e_k}[\mu^k]<+\infty$, there exists a sequence of underlying deformations, $\{\bm f^k\}\subset\mathcal{A}_{\e_k}$ such that $\mu^k=\delta_{\nabla_{\varepsilon_k}\bm f^k}$. Then we have the following:
    \begin{enumerate}
        \item[(i)] There exists $\mu\in L^\infty_w(\Omega_1,\mathcal{M}(\RR^{3\times3}))$ and a (not relabelled) subsequence $\{\mu^k\}$ such that
        \begin{align*}
            \mu^k\xrightharpoonup{*}\mu\text{ in }L^\infty_w(\Omega_1,\mathcal{M}(\RR^{3\times 3})),
        \end{align*}
        and
        \begin{align*}
            \nu:=\mathsf{R}\mu\in\mathcal{JY}^{p,-q}(\Omega,\RR^{3\times2}),
        \end{align*}
        with $q=\frac{ps}{p+s}>0$.

        \item[(ii)] There exists $\bm f\in W^{1,p}(\Omega_1,\RR^3)$ such that for a subsequence of underlying deformations (not relabelled), $\{\bm f^k\}$ we have
        \begin{align*}
            \bm f^k\weakarrow\bm f\text{ in }W^{1,p}(\Omega_1,\RR^3)\\
            \nabla_{\e_k}\bm f^k\weakarrow(\bar{\nabla}\bm f|\bm d)\text{ in }L^p(\Omega_1,\RR^{3\times 3}),
        \end{align*}
        where $\bar{\nabla}=\left(\diffp{{}}{{x_1}},\diffp{{}}{{x_2}}\right)$ is the 2D gradient and $\bm d\in L^p(\Omega_1,\RR^3)$, referred to as a \textit{Cosserat vector}, cannot be identified in terms of the limit deformation $\bm f$. Moreover,
        \begin{align*}
            \innerpdt{\nu,\text{id}}=\bar{\nabla}\bm f.
        \end{align*}

        \item[(iii)] In addition to the above convergence, if $p>2$, we have
        \begin{align*}
            \bm f_{,1}^k\times\bm f_{,2}^k\weakarrow\bm f_{,1}\times\bm f_{,2}\text{ in }L^{p/2}(\Omega_1,\RR^3)
        \end{align*}
        and if $p>3$, then
        \begin{align*}
            \det\nabla_{\e_k}\bm f^k\weakarrow\alpha>0\text{ in }L^{p/3}(\Omega_1)
        \end{align*}
    \end{enumerate} 
\end{proposition}
\begin{proof}
    From~\eqref{itm:dimension-reduction-growth}, the bound $\abs{\nabla\bm f^k}\le\abs{\nabla_{\e_k}\bm f^k}$, and the Poincar\'e inequality (recall that we assume $\fint_{\Omega_1}\bm f^k\dif x=0$), we have
    \begin{align}\label{eqn:dimension-reduction-coercivity-estimate}
        \mathcal{E}_{\e_k}[\mu^k] &= \int_{\Omega_1} W(x', \nabla_{\e_k} \bm f^k) \ dx \nonumber\\
				&\geq C_1\left(\norm{\nabla_{\e_k} \bm f^k}^p_{L^p(\Omega_1,\RR^{3\times 3})} +  \norm{\frac{1}{\det \nabla_{\e_k} \bm f^k}}^s_{L^s(\Omega_1)}  - C_2\right) \nonumber\\
				&\geq C_1\left(\norm{\nabla \bm f^k}^p_{L^p(\Omega_1,\RR^{3\times 3})} + \norm{\frac{1}{\det \nabla_{\e_k} \bm f^k}}^s_{L^s(\Omega_1)}  - C_2\right) \nonumber\\
				&\geq C\norm{ \bm f^k}^p_{W^{1,p}(\Omega_1,\RR^3)} + C_1 \norm{\frac{1}{\det \nabla_{\e_k} \bm f^k}}^s_{L^s(\Omega_1)}  - C_1C_2.
    \end{align}
    Since $p,s>1$, there exist $\bm D\in L^p(\Omega_1,\RR^{3\times 3})$ and $\beta\in L^s(\Omega_1)$, such that up to subsequences,
    \begin{align*}
        \nabla_{\e_k}\bm f^k\weakarrow\bm D&\text{ in }L^p(\Omega_1,\RR^{3\times 3})\\
        \frac{1}{\det\nabla_{\e_k}\bm f^k}=\frac{\e_k}{\det\nabla\bm f^k}\weakarrow\beta&\text{ in }L^s(\Omega_1).
    \end{align*}
    Moreover, since $\sup_k\norm{\bm f^k}^p_{W^{1,p}(\Omega_1,\RR^3)}<+\infty$ we have (up to subsequences),
    \begin{align*}
        \bm f^k\weakarrow\bm f\text{ in }W^{1,p}(\Omega_1,\RR^3)&\implies\bm f^k\rightarrow\bm f\text{ in }L^p(\Omega_1,\RR^3)\text{ by compact embedding,}\\
        \frac{1}{\e_k}\bm f^k_{,3}\weakarrow\bm d\text{ in }L^p(\Omega_1,\RR^3)&\implies
        \bm f^k_{,3}\rightarrow 0\text{ in }L^p(\Omega_1,\RR^3)
    \end{align*}
    Thus, we deduce that $\bm f$ is independent of $x_3$, i.e. $\bm f\in W^{1,p}(\Omega,\RR^3)$. Also,
    \begin{align*}
        \bm D = (\bar{\nabla}\bm f|\bm d),
    \end{align*}
    where $\bm d\in L^p(\Omega_1,\RR^3)$ cannot be identified in terms of $\bm f$.

    From the Banach-Alaoglu theorem, we have
    \begin{align*}
        \mu^k\xrightharpoonup{*}\mu\text{ in }L^\infty_w(\Omega_1,\mathcal{M}(\RR^{3\times3})).
    \end{align*}
    From continuity, 
    \begin{align*}
        \mathsf{R}\mu^k \xrightharpoonup{*} \mathsf{R}\mu =: \nu \text{ in } L_w^{\infty}(\Omega, \mathcal{M}(\RR^{3\times 2})).
    \end{align*} 
    It can be shown that $\nu$ is a $p$-gradient Young measure by manually checking the conditions of Theorem~\ref{thm:kinderlehrer-pedregal}. We omit the computation which is carried out in~\cite[Lemma 8.1]{freddi_energy_2004}. We now show that $\nu\in \mathcal{JY}^{p,-q}(\Omega, \RR^{3\times 2})$. Since $\nabla_{\e_k}\bm f^k=\left(\bar{\nabla}\bm f^k|\frac{1}{\e_k}\partial_3\bm f^k\right)$ and $\bar\nabla\bm f^k = \mathsf{P}\nabla\bm f^k$, the definition of $W_0$ and Proposition~\ref{prop:W0-growth-condition} give
    \begin{align*}
        C>\int_{\Omega_1} W(x', \nabla_{\e_k} \bm f^k) \dif x&\geq\int_{\Omega_1} W_0(x', \mathsf{P}\nabla\bm f^k) \dif x\\
        &\geq C_1\left(c_p\int_{\Omega_1}\Phi(\mathsf{P}\nabla\bm f^k)\dif x-C_2\right),
    \end{align*}
    where we recall that $\Phi(\bar{\bm F})=\abs{\bar{\bm F}}^p+\abs{J(\bar{\bm F})}^{-q}$. Since $\{\nabla_{\e_k}\bm f^k\}$ generates $\mu$ and $\int\Phi(\mathsf{P}\bm F)\dif \mu_x^k=\int\Phi\dif(\mathsf{P}_{\#}\mu_x^k)$, we conclude that 
    \begin{align*}
        \int_{\Omega_1}\int_{\RR^{3\times 2}} \Phi(\bm F) \dif(\mathsf{P}_{\#}\mu^k_x)(\bm F)\dif x\leq\tilde{C}<+\infty,
    \end{align*}
    for some constant $\tilde{C}>0$. Taking the $\liminf$ of the above and using the continuity of the projection/average operators and Theorem~\ref{thm:fundamental-property-of-Young-measures},
    \begin{align*}
        +\infty&>\liminf_{k\rightarrow +\infty}\int_\Omega\int_{-1/2}^{1/2}\int_{\RR^{3\times 2}}\Phi(\bm F)\dif(\mathsf{P}_{\#}\mu^k_{(x',x_3)})(\bm F)\ d x_3\ d x' \\
        &=\liminf_{k\rightarrow +\infty}\int_\Omega\int_{\RR^{3\times 2}}\Phi(\bm F)\dif(\mathsf{R}\mu^k)_{x'}(\bm F)\ d x'\\
        &\geq \int_\Omega\int_{\RR^{3\times 2}}\Phi(\bm F)\dif\nu_{x'}(\bm F)\ dx'.
    \end{align*}

    It remains to show that $\langle \nu,\id\rangle=\overline{\nabla}\bm f$. This follows from the relation
    \begin{align*}
        \langle\mathsf{P}_\#\mu^k_{(x',x_3)},\id\rangle=\langle\mu^k_{(x',x_3)},\id\circ\mathsf{P}\rangle=\int_{\RR^{3\times3}}\mathsf{P}(A)\ d(\delta_{\nabla_{\e_k} \bm f^k})(A)=\overline{\nabla}\bm f^k.
    \end{align*}
    Sending $k\rightarrow \infty$ in the above equation, we get $\langle\mathsf{P}_\#\mu_{(x',x_3)},\id\rangle=\overline{\nabla}\bm f$ (see Lemma~\ref{lem:weak-convergence-barycenter}). Averaging this in $x_3$ gives the barycenter condition. This concludes the proof of (i) and (ii).

    We turn our attention to (iii). The convergence of $\left\{\bm f_{,1}^k\times\bm f_{,2}^k\right\}$ follows easily from the observation that each component of the vector is a $2\times 2$ subdeterminant of the form~\eqref{eqn:subdeterminants3x2} of the sequence $\bar{\nabla}\bm f^k$. Since $p>2$, we get weak convergence in $L^{p/2}$ from~\cite[Lemma 1.14]{dacorogna_direct_2007}.
    
    For a matrix $\bm F\in\RR^{3\times3}$ with singular values $\sigma_1,\sigma_2,\sigma_3$ we have from the AM-GM inequality, 
    \begin{align*}
        \abs{\det\bm F}^{1/3}=\left(\sigma_1\sigma_2\sigma_3\right)^{1/3}\leq\frac{\sigma_1+\sigma_2+\sigma_3}{3}\leq\frac{\abs{\bm F}}{\sqrt{3}}.
    \end{align*}
    Thus we have $\abs{\det\nabla_{\e_k}\bm f^k}^{p/3}\leq C\abs{\nabla_{\e_k}\bm f^k}^p$ for every $k$. Hence, $\sup_k\norm{\det\nabla_{\e_k}\bm f^k}_{L^{p/3}}^{p/3}<+\infty$ and if $p>3$, there exists $\alpha\in L^{p/3}(\Omega_1)$ such that, up to subsequences,
    \begin{align*}
        \det\nabla_{\e_k}\bm f^k\weakarrow\alpha\text{ in }L^{p/3}(\Omega_1).
    \end{align*}
    Since $\det\nabla_{\e_k}\bm f^k>0$ a.e., we conclude that $\alpha\geq0$. From~\eqref{eqn:dimension-reduction-coercivity-estimate}, we also have $\sup_{k}\norm{\frac{1}{\det\nabla_{\e_k}\bm f^k}}_{L^s(\Omega_1)}^s<+\infty$ and thus for $r:=\min\{s,p/3\}$,
    \begin{align*}
        \sup_{k}\int_{\Omega_1}\frac{1}{(\det\nabla_{\e_k}\bm f^k)^r}+(\det\nabla_{\e_k}\bm f^k)^r\dif x<+\infty.
    \end{align*}
    Since $t\mapsto t^r+t^{-r}$ is strictly convex for $t>0$, we get
    \begin{align*}
        +\infty>\liminf_{k\rightarrow\infty}\int_{\Omega_1}\frac{1}{(\det\nabla_{\e_k}\bm f^k)^r}+(\det\nabla_{\e_k}\bm f^k)^r\dif x\geq\int_{\Omega_1}\frac{1}{\alpha^r}+\alpha^r\dif x.
    \end{align*}
    Thus, $\alpha>0$ almost everywhere in $\Omega_1$.
\end{proof}

We proceed to prove Theorem~\ref{thm:gc}. Since compactness (item (i)) was already proved, we will prove only the liminf and limsup inequalities (items (ii) and (iii)).

\subsection{\texorpdfstring{$\Gamma$}{Gamma}-convergence}

\begin{proof}[Proof of Theorem \ref{thm:gc}]
       		\textit{(ii) Liminf inequality}
 We may assume that the left-hand side of the inequality is finite and that (by passing through a subsequence, but not relabeling) the $\liminf$ is a limit for the sequence $\{\mu_{n}\}$. Then, $\sup_n \mathcal{E}_{\e_n}[\mu_{n}]<+\infty$, and from Proposition~\ref{thm:Young-measure-compactness} there exists $\mu\in L_w^\infty(\Omega_1,\mathcal{M}(\bbR^{3\times3}))$ with (not relabelled subsequence) $\mu_{n}\weakstar\mu$ weakly* in $L^\infty_w(\Omega_1,\mathcal{M}(\bbR^{3\times3}))$. Weak* convergence and the continuity of $\mathsf{R}$ imply that $\nu=\mathsf{R}(\mu)\in \mathcal{JY}^{p,-q}(\Omega,\bbR^{3\times 2})$. The underlying deformations also converge weakly:
\[
    \bm f_{n}\weak \bm f\equiv \bm f(x')\quad\text{in }W^{1,p}(\Omega_1,\bbR^3).
\] The liminf inequality follows from a similar computation as in the proof of compactness in Proposition~\ref{thm:Young-measure-compactness}.

\textit{(iii) Recovery sequence}
Using Proposition~\ref{prop:smooth-YM-generation} we obtain a sequence $\{\bm u^k\}_{k=1}^\infty\subset C^{\infty}(\bar{\Omega},\RR^3)$ with $\{\nabla \bm u^k\}$  $p$-equiintegrable and satisfying~\eqref{eq:Cinfty-Phi-equiintergability}. Our bulk energy $W(x',\cdot)$ satisfies the hypotheses in~\cite{anza_nonlinear_2008} for a.e. $x$, so we adapt some of the arguments presented there. We define the set-valued function
\begin{align*}
    \Lambda_k^j(x'):=\left\{\bm\xi\in\RR^3:\det(\nabla \bm u^k(x')|\bm\xi)\geq\frac{1}{j}\right\}.
\end{align*}
Note that since $\bm u^k$ is smooth on $\bar{\Omega}$, there exists $\eta_k>0$ such that $J(\nabla \bm u^k)\geq\eta_k$ (extreme value theorem): $\Lambda^j_k(x')$ is thus non-empty for all $x'\in\bar{\Omega}$.

Define
\begin{align*}
    \Upsilon^j_k(x'):=\inf_{\bm \xi\in\Lambda_k^j(x')}W(x',(\nabla \bm u^k(x')|\bm \xi)).
\end{align*}

For fixed $k$, $x'\in\bar{\Omega}$, $\Lambda^j_k(x')$ is a monotonic sequence of sets, therefore, $\Upsilon^1_k\geq\Upsilon^2_k\geq...\geq 0$. Since $\bigcup_{j=1}^\infty\Lambda_k^j(x')=\left\{\bm \xi\in\RR^3:\det(\nabla \bm u^k(x')|\bm\xi)> 0\right\}$, we have pointwise
\begin{align*}
    \lim_{j\rightarrow\infty}\Upsilon^j_k(x')=\inf_{\bm\xi:\det(\nabla\bm u^k|\bm\xi)>0}W(x',(\nabla\bm u^k(x')|\bm \xi))=W_0(x',\nabla\bm u^k(x')),
\end{align*}
by the definition of $W_0$. Clearly $\bm m^k(x')=\frac{\partial_1\bm u^k\times\partial_2\bm u^k}{\abs{\partial_1\bm u^k\times\partial_2\bm u^k}^2}\in\Lambda_k^1(x')$, since $\det(\nabla\bm u^k|\bm m^k)=1$. We have $\abs{(\bar{\bm F}|\bm z)}^p\leq C_p\left(\abs{\bar{\bm F}}^p+\abs{\bm z}^p\right)$, and $\abs{\bm m^k}=J(\nabla\bm u^k)^{-1}$, so hypothesis~\ref{itm:dimension-reduction-growth} gives
\begin{align*}
    \Upsilon^1_k(x')\leq W(x',(\nabla\bm u^k(x')|\bm m^k(x')))&\leq C_3\left(\abs{(\nabla\bm u^k|\bm m^k)}^p+\frac{1}{\abs{\det(\nabla\bm u^k|\bm m^k)}^s}+C_4\right)\\
    &\leq C\left(\abs{\nabla\bm u^k}^p+\abs{J(\nabla\bm u^k)}^{-p}\right) + C\\
    &\leq C\left(\abs{\nabla\bm u^k}^p+\eta_k^{-p}\right) + C.
\end{align*}
Here $C>0$ changes from occurrence to occurrence. The right hand side above is integrable, so from the dominated convergence theorem and the monotonicity of $\Upsilon^j_k$ we have
\begin{align*}
\lim_{j\rightarrow\infty}\int_\Omega\Upsilon^j_k(x')\dif x'=\int_\Omega W_0(x',\nabla\bm u^k(x'))\dif x'.
\end{align*}
Now, pick $j_k$ large enough so that
\begin{align}\label{eq:recovery-sequence-interchange-step-1}
    \int_\Omega\Upsilon^{j_k}_k(x')\dif x'\leq\int_\Omega W_0(x',\nabla\bm u^k(x'))\dif x'+\frac{1}{2k}.
\end{align}
From the continuous selection lemma~\cite[Lemma 3.2]{anza_nonlinear_2008}, we can interchange the infimum and integral in the following way: 
\begin{align*}
    \inf_{\bm z\in C(\bar{\Omega},\Lambda^{j_k}_k)}\int_\Omega W(x',(\nabla \bm u^k(x')|\bm z(x')))\dif x' = \int_\Omega\Upsilon_k^{j_k}(x')\dif x'.
\end{align*}
Choose $\bm z^k(x')\in C(\bar{\Omega},\Lambda^{j_k}_k)$, so that
\begin{align*}
    \int_\Omega W(x',(\nabla\bm u^k(x')|\bm z^k(x')))\dif x'\leq\int_\Omega\Upsilon^{j_k}_k(x')\dif x'+\frac{1}{2k}.
\end{align*}
Combining the above with~\eqref{eq:recovery-sequence-interchange-step-1}:
\begin{align}\label{eq:recovery-sequence-interchange-step-2}
    \int_{\Omega} W(x',(\nabla\bm u^k(x')|\bm z^k(x')))\dif x'\leq\int_\Omega W_0(x',\nabla\bm u^k(x'))\dif x'+\frac{1}{k}.
\end{align}
Let $\{\tilde{\bm z}^k_l\}_{l=1}^\infty\subset C^\infty(\bar{\Omega},\RR^3)$ approximate $\bm z^k$ uniformly. Note that
\begin{align*}
    \det(\nabla\bm u^k|\tilde{\bm z}^k_l) &= \det(\nabla\bm u^k|\bm z^k)+\det(\nabla\bm u^k|\tilde{\bm z}^k_l-\bm z^k)\\
    &\geq\frac{1}{j_k}-\norm{J(\nabla\bm u^k)}_{L^\infty}\norm{\tilde{\bm z}^k_l-\bm z^k}_{L^\infty}.
\end{align*}
Since $\tilde{\bm z}^k_l\rightarrow\bm z^k$ uniformly, we have $\norm{\tilde{\bm z}^k_l-\bm z^k}_{L^\infty}\leq\left(2j_k\norm{J(\nabla\bm u^k)}_{L^\infty}\right)^{-1}$ for all $l$ large enough, and hence
\begin{align*}
    \det(\nabla\bm u^k|\tilde{\bm z}^k_l)\geq\frac{1}{2j_k}
\end{align*}
for such $l$. The growth condition~\eqref{itm:dimension-reduction-growth} then bounds $W(x',(\nabla\bm u^k|\tilde{\bm z}^k_l))$ uniformly in $l$ and $x'$, so by the continuity of $W(x',\cdot)$ for a.e. $x'\in\Omega$ and dominated convergence we may fix $l_k$ with
\begin{align*}
    \abs{\int_\Omega W(x',(\nabla\bm u^k|\tilde{\bm z}^k_{l_k}))\dif x'-\int_\Omega W(x',(\nabla\bm u^k|\bm z^k))\dif x'}\leq\frac{1}{k}.
\end{align*}
To simplify notation, let $\bm b^k:=\tilde{\bm z}^k_{l_k}$. Combining with~\eqref{eq:recovery-sequence-interchange-step-2}, we get
\begin{align}
    \label{eq:recovery-sequence-energy-estimate-1}
    \int_\Omega W(x',(\nabla\bm u^k|\bm b^k))\dif x'\leq\int_\Omega W_0(x',\nabla\bm u^k)\dif x'+\frac{2}{k}.
\end{align}
Let $\e_n\rightarrow0$ and define $\bm y^k_{\e_n}(x',x_3)=\bm u^k(x')+\e_n x_3\bm b^k(x')$ so that
\begin{align*}
    \nabla_{\e_n}\bm y^k_{\e_n}(x',x_3)=\left(\nabla\bm u^k(x')+\e_nx_3\nabla\bm b^k(x')|\bm b^k(x')\right).
\end{align*}
Notice that as $n\rightarrow\infty$,
\begin{align*}
    \nabla_{\e_n}\bm y^k_{\e_n}\rightarrow\left(\nabla\bm u^k|\bm b^k\right)\text{ uniformly.}
\end{align*}
Computing the determinant of $\nabla_{\e_n}\bm y^k_{\e_n}$:
\begin{align*}
    \det(\nabla_{\e_n}\bm y^k_{\e_n}) &= \left[(\partial_1\bm u^k+\e_nx_3\partial_1\bm b^k)\times(\partial_2\bm u^k+\e_nx_3\partial_2\bm b^k)\right]\cdot\bm b^k\\
    &=\det(\nabla\bm u^k|\bm b^k)+\e_n x_3(\partial_1\bm u^k\times\partial_2\bm b^k+\partial_1\bm b^k\times\partial_2\bm u^k)\cdot\bm b^k+(\e_nx_3)^2\det(\nabla\bm b^k|\bm b^k).
\end{align*}
The first term is bounded below by $\frac{1}{2j_k}$ and the second and third terms are bounded (since $\bm u^k$ and $\bm b^k$ are $C^\infty$ on $\bar{\Omega}$),
\begin{align*}
    \abs{x_3\e_n(\partial_1\bm u^k\times\partial_2\bm b^k+\partial_1\bm b^k\times\partial_2\bm u^k)\cdot\bm b^k+(\e_n x_3)^2\det(\nabla\bm b^k|\bm b^k)}\leq M_k\e_n.
\end{align*}
Hence, taking $n$ large enough that $\e_n\leq(4j_kM_k)^{-1}$,
\begin{align*}
    \det(\nabla_{\e_n}\bm y^k_{\e_n})\geq\frac{1}{2j_k}-M_k\e_n\geq\frac{1}{4j_k},
\end{align*}
so by~\eqref{itm:dimension-reduction-growth} the family $\{W(x',\nabla_{\e_n}\bm y^k_{\e_n})\}_n$ is bounded above uniformly in $n$ and $x$. Thus from dominated convergence,
\begin{align}\label{eq:recovery-sequence-energy-estimate-2}
    \lim_{n\rightarrow\infty}\int_{\Omega_1}W(x',\nabla_{\varepsilon_n}\bm y^k_{\e_n})\dif x=\int_\Omega W(x',(\nabla\bm u^k|\bm b^k))\dif x'.
\end{align}
We can now choose a diagonal sequence with the following scheme: define $k_1=1$ and
\begin{align*}
    k_{n+1}=\begin{cases}
        k_n+1\quad\text{if~\eqref{eq:retardation-sequence-condition} is true}\\
        k_n\quad\text{if not.}
    \end{cases}
\end{align*}
where
\begin{align}\label{eq:retardation-sequence-condition}
    \begin{cases}
        \int_{\Omega_1}W(x',\nabla_{\e_n}\bm y^{k_n}_{\e_n})\dif x\leq\int_\Omega W(x',(\nabla\bm u^{k_n}|\bm b^{k_n}))\dif x'+\frac{1}{k_n}\\
        \text{ and }\e_n\norm{\nabla \bm b^{k_n}}_{L^\infty}\leq\frac{1}{k_n}
    \end{cases}
\end{align}
The latter condition is required later when we consider Young measure convergence and since $\nabla\bm b^k$ is bounded on $\bar{\Omega}$, the condition can always be met.
From~\eqref{eq:recovery-sequence-energy-estimate-2}, it follows that $k_n\rightarrow\infty$. Then setting $\bm y_n:=\bm y^{k_n}_{\e_n}$ it follows from~\eqref{eq:recovery-sequence-energy-estimate-1} that
\begin{align*}
    \int_{\Omega_1}W(x',\nabla_{\e_n}\bm y_n)\dif x\leq\int_{\Omega}W_0(x',\nabla\bm u^{k_n})\dif x'+\frac{3}{k_n}.
\end{align*}
Since $\Phi(\nabla\bm u^{k_n})$ is equiintegrable, from Proposition~\ref{prop:W0-growth-condition}, so is $W_0(x',\nabla\bm u^{k_n})$. Taking the limsup as $n\rightarrow\infty$ gives
\begin{align*}
    \limsup_{n\rightarrow\infty}\int_{\Omega_1}W(x',\nabla_{\e_n}\bm y_n)\dif x\leq\lim_{n\rightarrow\infty}\int_{\Omega}W_0(x',\nabla\bm u^{k_n})\dif x'=\int_\Omega\int_{\RR^{3\times2}}W_0(x',\bm F)\dif\nu_{x'}(\bm F)\dif x'.
\end{align*}

Now set $\mu_n:=\delta_{\nabla_{\e_n}\bm y_n}$. We claim that $\mathsf{R}\mu_n\xrightharpoonup{*}\nu$ as measures. By definition, for $\phi\in C_c(\RR^{3\times2})$ and $g\in L^1(\Omega)$
\begin{align*}
    \int_\Omega\langle\mathsf{R}\mu_n,\phi\rangle g(x')\dif x'=\int_\Omega\int_{-1/2}^{1/2}\phi(\nabla\bm u^{k_n}+\e_n x_3\nabla\bm b^{k_n})g(x')\dif x_3\dif x'.
\end{align*}
Let $\omega_\phi$ denote the modulus of continuity of $\phi$. Then,
\begin{align*}
    &\int_\Omega\int_{-1/2}^{1/2}\abs{\phi(\nabla\bm u^{k_n}+\e_n x_3\nabla\bm b^{k_n})-\phi(\nabla\bm u^{k_n})}g(x')\dif x_3\dif x'\\
    &\leq\int_\Omega\int_{-1/2}^{1/2}\omega_{\phi}\left(\e_n\abs{x_3\nabla\bm b^{k_n}}\right)\abs{g(x')}\dif x_3\dif x'\\
    &\leq \omega_\phi\left(\frac{\e_n\norm{\nabla\bm b^{k_n}}_{L^\infty}}{2}\right)\norm{g}_{L^1}
\end{align*}
and from~\eqref{eq:retardation-sequence-condition}, the right hand side approaches zero as $n\rightarrow\infty$. Therefore,
\begin{align*}
    \int_\Omega\langle\mathsf{R}\mu_n,\phi\rangle g(x')\dif x'=\int_\Omega\int_{-1/2}^{1/2}\phi(\nabla\bm u^{k_n})g(x')\dif x_3\dif x'+o(1)\to \int_\Omega\langle\nu_{x'},\phi\rangle g(x')\dif x',
\end{align*}
i.e., $\mathsf{R}\mu_n\xrightharpoonup{*}\nu$ in the sense of measures.
\end{proof}

\section{Proof of Proposition \ref{prop:relation-to-LDR} and Corollary \ref{cor:linkrelax}}\label{sec:relaxation}
The membrane energy $\mathcal{E}_0$ we identified in the earlier section is defined on a class of Young measures. We will now relate this to an energy defined on deformations by exhibiting $\mathcal{E}_0$ as a relaxation. Consider the functional
\begin{align*}
    \mathcal{E}_{\text{el}}(\nu)=\begin{cases}
        \int_\Omega W_0(x',\nabla\bm u)\dif x' &\text{ if }\exists\nu=\delta_{\nabla\bm u}\\
        +\infty &\text{otherwise}
    \end{cases}
\end{align*}
which is defined on elementary Young measures. We then have the following relaxation result:
\begin{proposition}
    \label{thm:relaxation}
    The (sequential) lower semicontinuous envelope of $\mathcal{E}_{\text{el}}$ is $\mathcal{E}_0$, i.e.,
    \begin{align*}
        \mathcal{E}_0[\nu]=\min\left\{\liminf_{k\rightarrow\infty}\mathcal{E}_{\text{el}}[\nu^k]:\nu^k\xrightharpoonup{*}\nu\text{ in }L^\infty_w(\Omega,\mathcal{M}(\RR^{3\times2}))\right\}.
    \end{align*}
\end{proposition}
\begin{proof}
    Liminf: Suppose $\delta_{\nabla\bm u^k}\xrightharpoonup{*}\nu$ with $\liminf_{k\rightarrow\infty}\mathcal{E}_{\text{el}}[\delta_{\nabla\bm u^k}]<\infty$ (otherwise there is nothing to prove). Then, from~\eqref{eqn:W0-growth-condition}, the sequence $\{\bm u^k\}$ is weakly precompact and therefore a subsequence $\bm u^k\weakarrow\bm u$ in $W^{1,p}$ with $[\nu]=\nabla\bm u$. Since $\nu$ is generated by gradients, we have $\nu\in\mathcal{GY}^p$ and from Theorem~\ref{thm:fundamental-property-of-Young-measures} and the growth conditions on $W_0$:
    \begin{align*}
        \int_\Omega\innerpdt{\nu_{x'},W_0(x',\cdot)}\dif x'\leq\liminf_k\int_\Omega W_0(x',\nabla\bm u^k)\dif x',
    \end{align*}
    i.e., $\nu\in\mathcal{JY}^{p,-q}$ and $\mathcal{E}_0[\nu]\leq\liminf_k\mathcal{E}_{\text{el}}[\nu^k]$.

    Limsup: Now suppose $\nu\in\mathcal{JY}^{p,-q}$. From Proposition~\ref{prop:smooth-YM-generation}, there exists a sequence of smooth functions $\{\bm u^k\}\subset C^\infty(\Omega,\RR^3)$ generating $\nu$ such that $\{\Phi(\nabla\bm u^k)\}$ is equiintegrable. Using Proposition~\ref{prop:W0-growth-condition}, $\{W_0(\cdot,\nabla\bm u^k)\}$ is equiintegrable as well, thus
    \begin{align*}
        \lim_{k\rightarrow\infty}\mathcal{E}_{\text{el}}[\delta_{\nabla\bm u^k}]=\lim_{k\rightarrow\infty}\int_{\Omega}W_0(x',\nabla\bm u^k)\dif x'=\int_\Omega\innerpdt{\nu_{x'},W_0}\dif x'=\mathcal{E}_0[\nu].
    \end{align*}
\end{proof}

We proceed to the proof of Proposition \ref{prop:relation-to-LDR}.

\begin{proof}[Proof of Proposition \ref{prop:relation-to-LDR}]
    Since $QW_0$ is the quasiconvexification, $W_0\geq QW_0$. From~\cite[Proposition A.1(iii), A.3 and A.5]{hafsa2006nonlinear}, $QW_0$ is continuous and $QW_0(\cdot,\bar{\bm{F}})\leq C_1(1+\abs{\bar{\bm{F}}}^p)$ for some constant $C_1$. Let $\nu\in\mathcal{JY}^{p,-q}$ with $[\nu]=\nabla\bm u$. The lower bound then follows from applying the Jensen inequality for gradient Young measures:
    \begin{align*}
        \mathcal{E}_0[\nu]\geq\int_\Omega \innerpdt{\nu,QW_0(x',\cdot)}\dif x'\geq\int_\Omega QW_0(x',\nabla\bm u)\dif x'= E_{\text{LDR}}[\bm u].
    \end{align*}

    For the upper bound, we use the $\Gamma$-convergence result of~\cite{anza_nonlinear_2008} which is formulated in terms of the averaging operator
    \begin{align*}
        \pi_\varepsilon(\bm u)=\frac{1}{\varepsilon}\int_{-\varepsilon/2}^{\varepsilon/2}\bm u(x',x_3)\dif x_3.
    \end{align*}
    Let $\{\bm u_\varepsilon\}$ be a recovery sequence for $\bm u$ given by~\cite[Theorem 2.5]{anza_nonlinear_2008} i.e., $\pi_\varepsilon(\bm u_\varepsilon)\rightarrow\bm u$ in $L^p$, such that $\lim_{\varepsilon\rightarrow0}E^{3D}_\varepsilon[\bm u_\varepsilon]= E_{\text{LDR}}[\bm u]$ (recalling $E^{3D}_\varepsilon$ is the unrescaled bulk energy).
    
    The growth conditions imply $\det\nabla\bm u_\varepsilon>0$. Then, rescaling the domain to unit thickness, we have a sequence $\{\bm f_\varepsilon\}\subset W^{1,p}(\Omega_1,\RR^3)$ with $\bm f_\varepsilon\in \mathcal{A}_{\varepsilon}$. Then taking $\mu_\varepsilon=\delta_{\nabla_\e\bm f_\varepsilon}$ and applying the compactness result (Proposition~\ref{thm:Young-measure-compactness}), we have $\mathsf{R}\mu_{\varepsilon}\xrightharpoonup{*}\nu\in\mathcal{JY}^{p,-q}$ with $[\nu]=\nabla\bm f$ and $\bm f_\varepsilon\weakarrow\bm f$ weakly in $W^{1,p}(\Omega_1,\RR^3)$ with $\partial_3\bm f_\varepsilon\rightarrow0$ strongly in $L^p$ ($\bm f$ does not depend on $x_3$). We can identify the limits $\bm f$ and $\bm u$ since $\pi_1(\bm f_\varepsilon)=\pi_\varepsilon(\bm u_\varepsilon)\rightarrow\bm u$. From the liminf inequality of Theorem~\ref{thm:gc},
    \begin{align*}
        \mathcal{E}_0[\nu]\leq\liminf_{\varepsilon\rightarrow0}\mathcal{E}_\varepsilon[\mu_\varepsilon]=\lim_{\varepsilon\rightarrow0}E_\varepsilon[\bm f_\varepsilon]=  E_{\text{LDR}}[\bm u].
    \end{align*}

    Since $\nu\in\mathcal{JY}^{p,-q}$ with $[\nu]=\nabla\bm u$ is admissible, and $\mathcal{E}_0[\lambda]\geq  E_{\text{LDR}}[\bm u]$ for any admissible $\lambda$, we conclude that $\nu$ attains the infimum thus $\mathcal{M}_{\bm u}\neq\emptyset$.
\end{proof}
As a consequence of the above proposition, we can characterize the limiting behavior of sequences of 3D deformations that are asymptotically optimal:

\begin{proof}[Proof of Corollary \ref{cor:linkrelax}]
    Since $\mathcal{E}_0$ is finite on $\mathcal{JY}^{p,-q}$, $ E_{\text{LDR}}[\bm u]<+\infty$ and hence $\sup_\varepsilon \mathcal{E}_\varepsilon[\mu_\varepsilon]<+\infty$ for $\mu_{\varepsilon}:=\delta_{\nabla_\varepsilon\bm f_\varepsilon}$. Precompactness then follows from Proposition~\ref{thm:Young-measure-compactness}.

    Now suppose $\nu$ is a limit point $\{\mathsf{R}\mu_{\varepsilon_j}\}$ along $\varepsilon_j\rightarrow0$. By Proposition~\ref{thm:Young-measure-compactness}(i), $\nu\in\mathcal{JY}^{p,-q}(\Omega,\RR^{3\times2})$, and by Proposition~\ref{thm:Young-measure-compactness}(ii), $\innerpdt{\nu,\id}=\bar{\nabla}\bm u$, hence $\nu\in\mathcal{A}_{\bm u}$. From the liminf inequality of Theorem~\ref{thm:gc},
    \begin{align*}
        \mathcal{E}_0[\nu]\leq\liminf_{j\rightarrow\infty}\mathcal{E}_{\varepsilon_j}[\mu_{\varepsilon_j}]=\liminf_{j\rightarrow\infty}E_{\varepsilon_j}[\bm f_{\varepsilon_j}]\leq E_{\text{LDR}}[\bm u].
    \end{align*}
    On the other hand, $\nu\in\mathcal{A}_{\bm u}$ and Proposition~\ref{prop:relation-to-LDR} gives $\mathcal{E}_0[\nu]\geq E_{\text{LDR}}[\bm u]$. Hence $\mathcal{E}_0[\nu]= E_{\text{LDR}}[\bm u]=\min_{\mathcal{A}_{\bm u}}\mathcal{E}_0$, i.e., $\nu\in\mathcal{M}_{\bm u}$. Also, it follows that $\liminf_{j\rightarrow\infty}E_{\varepsilon_j}[\bm f_{\varepsilon_j}]= E_{\text{LDR}}[\bm u]$. Using this in conjunction with the hypothesis $\limsup_{\varepsilon\rightarrow0}E_\varepsilon[\bm f_\varepsilon]\leq E_{\text{LDR}}[\bm u]$, we conclude that $\lim_{\varepsilon\rightarrow0}E_\varepsilon[\bm f_\varepsilon]= E_{\text{LDR}}[\bm u]$.
\end{proof}

\section*{Acknowledgements}
The authors would like to thank Ian Tobasco for pointing out Brehm’s theorem as a possible tool for constructing the replacement maps and Barbora Bene\v{s}ov\'a for suggesting the Gromov–Eliashberg theorem.

Parts of this work appeared in Chapter~5 of GGN's PhD thesis~\cite{nair2024convexity}. GGN thanks Tim Healey for the invitation to join him on his sabbatical in Pisa, where this project began, and for financial support through NSF grant DMS-2006586. GGN also thanks Roberto Paroni for his hospitality during his visit to the University of Pisa.

MPS has been supported by the European Union-Next Generation EU through the PRIN2022 project ``NutShell''.
MPS	also acknowledges the Italian National Group of Mathematical Physics
INdAM-GNFM.

DPG was supported by the Zuckerman STEM leadership program, and by the FWF ESPRIT fellowship ESP 3828725.

\section*{AI disclosure}
The contents of this paper, including the results, their proofs and ideas underlying them were developed by the authors. During the preparation of this manuscript the authors used OpenAI's GPT-5.6 Sol, GPT-6 Astra, and Anthropic's Claude Opus 5 for language editing, stylistic improvements and figure generation. Opus 5 was also used for simplifying and checking some of the computations. The authors reviewed and edited all AI-assisted output and take full responsibility for the content of the manuscript.

\bibliographystyle{alpha}
\bibliography{ref}
\end{document}